\documentclass[12pt,a4paper]{article}
\usepackage{amsthm,amsmath,amssymb,amsfonts,delarray}
\newcommand\bigzerol{\smash{\hbox{\large 0}}}
\newcommand\bigzerou{\smash{\lower.3ex\hbox{\large 0}}}
\usepackage{lscape}
\usepackage[dvipdfm]{graphicx}
\usepackage[round]{natbib}
\newtheorem{thm}[subsection]{Theorem}
\newtheorem{lem}[subsection]{Lemma}

\theoremstyle{definition}
\newtheorem{rem}[subsection]{Remark}
\author{
Keisuke YANO and Fumiyasu KOMAKI\\
\small Department of Mathematical Informatics,\\
\small Graduate School of Information Science and Technology,\\
\small The University of Tokyo\\
\small 7-3-1 Hongo, Bunkyo-ku, Tokyo 113-8656, JAPAN\\
\small \texttt{\{keisuke\_yano,komaki\}@mist.i.u-tokyo.ac.jp}
}
\date{}
\title{Information criteria for multistep ahead predictions}
\begin{document}
\maketitle
\begin{abstract}
We propose an information criterion for multistep ahead predictions.
It is also used for extrapolations.
For the derivation,
we consider multistep ahead predictions under local misspecification.
In the prediction,
we show that
Bayesian predictive distributions
asymptotically have smaller Kullback--Leibler risks
than
plug-in predictive distributions.
From the results,
we construct an information criterion for multistep ahead predictions
by using an asymptotically unbiased estimator 
of the Kullback--Leibler risk of Bayesian predictive distributions.
We show the effectiveness of the proposed information criterion
throughout the numerical experiments.
\end{abstract}

\section{Introduction}

Consider multistep ahead predictions as follows:
let $x^{(N)}=(x_{1},\ldots,x_{N})$ be data 
from distribution $p(x^{(N)})$
and 
let $y^{(M)}=(y_{1},\ldots,y_{M})$ be target variables
from distribution $q(y^{(M)})$.
We assume that sample size $M$ is given as 
the constant multiplication of sample size $N$,
i.e., we assume that $M=cN$.
We predict the distribution of the target variables on the basis of the data.
Here, distributions $p(x^{(N)})$ and $q(y^{(M)})$ may be different but
we assume that $x_{1},\ldots,x_{N},y_{1},\ldots,y_{M}$ are independent.

For the prediction,
we consider $m_{\mathrm{full}}$ parametric models of the distributions of the data and the target variables
as follows:
for $m\in\{1,\ldots,m_{\mathrm{full}}\}$,
the $m$-th model $\mathcal{M}_{m}$ 
is given as $\{p_{m}(x^{(N)}|\theta_{m})q_{m}(y^{(M)}|\theta_{m}):\theta_{m}\in\Theta_{m}\}$.
Here,
$\Theta_{m}$ is a $d_{m}$-dimensional parametric space.
For simplicity,
we denote parameter $\theta_{m_{\mathrm{full}}}$ by $\omega$,
distribution $p_{m_{\mathrm{full}}}(x^{(N)}|\omega)$ by $p(x^{(N)}|\omega)$,
and
distribution $q_{m_{\mathrm{full}}}(y^{(M)}|\omega)$ by $q(y^{(M)}|\omega)$.
We denote parameter space $\Theta_{m_{\mathrm{full}}}$ by $\Theta$
and dimension $d_{m_{\mathrm{full}}}$ by $d_{\mathrm{full}}$.
After the model selection,
we construct the predictive distribution in the selected model.

As an example, consider the curve fitting.
We obtain the values of the unknown curve at points $(z_{1},\ldots,z_{i},\ldots,z_{N})$
and predict the distribution of the values at points $(z_{N+1},\ldots,z_{N+j},\ldots,z_{N+M})$.
We use regression models with the basis set
$\{\phi_{a}\}_{a=1}^{d_{\mathrm{full}}}$:
for $m\in\{1,\ldots,d_{\mathrm{full}}\}$,
for $i\in\{1,\ldots,N\}$,
and
for $j\in\{1,\ldots,M\}$,
the  $i$-th data and the $j$-th target variable in the $m$-th model are
given by
\begin{eqnarray*}
x_{i}=\mathop{\Sigma}_{a=1}^{m}\phi_{a}(z_{i})\theta_{m}^{a}+\epsilon_{i}
&\mathrm{and}&
y_{j}=\mathop{\Sigma}_{a=1}^{m}\phi_{a}(z_{N+j})\theta_{m}^{a}+\epsilon_{N+j},
\end{eqnarray*}
respectively.
Here, $\theta_{m}=(\theta_{m}^{1},\ldots,\theta_{m}^{m})$ represents an unknown vector.
Two random vectors
$\epsilon=(\epsilon_{1},\ldots,\epsilon_{N})^{\top}$
and 
$\tilde{\epsilon}=(\epsilon_{N+1},\ldots,\epsilon_{N+M})^{\top}$ are 
independent and distributed according to Gaussian distributions 
with mean zero and diagonal covariance matrices 
$\sigma^{2}I_{N\times N}$ and $\sigma^{2}I_{M\times M}$,
respectively.

We measure the performance of the predictive distribution $\hat{q}$ by the Kullback--Leibler risk:
\begin{eqnarray*}
R(p(\cdot)q(\cdot),\hat{q})&=&\int p(x^{(N)})\int q(y^{(M)})\log\frac{q(y^{(M)})}{\hat{q}(y^{(M)};x^{(N)})}
\mathrm{d}y^{(M)}\mathrm{d}x^{(N)}.
\end{eqnarray*}

In this paper,
we consider the asymptotics as the sample sizes $N$ and $M$ simultaneously go to infinity.
Note that since $M=cN$ we consider that $N$ goes to infinity.
We show that
for any smooth prior $\pi$,
the Bayesian predictive distribution $q_{m,\pi}(y^{(M)}|x^{(N)})$
in submodel $\mathcal{M}_{m}$
\begin{eqnarray}
q_{m,\pi}(y^{(M)}|x^{(N)})&=&
\frac{\int q_{m}(y^{(M)}|\theta_{m})p_{m}(x^{(N)}|\theta_{m})\pi(\theta_{m})
\mathrm{d}\theta_{m}}{\int p_{m}(x^{(N)}|\theta_{m})\pi(\theta_{m})\mathrm{d}\theta_{m}}
\end{eqnarray}
asymptotically has smaller Kullback--Leibler risk
than the plug-in predictive distribution
$q_{m}(y^{(M)}|\hat{\theta}_{m}(x^{(N)}))$
with the maximum likelihood estimator
in submodel $\mathcal{M}_{m}$.
Further,
the Kullback--Leibler risk of the Bayesian predictive distribution
varies according to the Fisher information matrices of the data and the target variables;
in the $i.i.d.$ settings,
the risk varies according to the multiplicative constant $c$.

From the results,
we construct an information criterion for the multistep ahead prediction
by using an asymptotically unbiased estimator 
of the Kullback--Leibler risk of the Bayesian predictive distribution.
Several numerical experiments show the performance of the proposed information criterion.

This paper is organized as follows:
in Section 2,
we prepare the notations and
state the assumptions to be used.
In Section 3,
we show that Bayesian predictive distributions
have smaller Kullback--Leibler risks than 
plug-in predictive distributions
in multistep ahead predictions.
In Section 4,
we propose information criteria for multistep ahead predictions.
By considering the variance of proposed information criteria,
we propose their bootstrap adjustments.
In Section 5,
we show two numerical experiments:
the curve fitting
and the normal regression model with an unknown variance.
In Section 6,
we present our conclusions.

\section{Notations and Assumptions}

We consider 
that the true distributions $p(x^{(N)})$ and $q(y^{(M)})$ 
belong to the full model $\mathcal{M}_{m_{\mathrm{full}}}$:
\begin{eqnarray*}
p(x^{(N)})=p(x^{(N)}|\omega^{*})
\,&\mathrm{and}&\,
q(y^{(M)})=q(y^{(M)}|\omega^{*}),
\end{eqnarray*}
where $\omega^{*}$ is a certain point in $\Theta$.
We refer to this parameter point $\omega^{*}$ as the true parameter point.

We consider
that the full model $\mathcal{M}_{m_{\mathrm{full}}}$
contains submodel $\mathcal{M}_{m}$.
Then,
we decompose the parameter $\omega$ in the full model $\mathcal{M}_{m_{\mathrm{full}}}$
into $\omega(\theta_{m},\gamma_{m})$.
We denote the parameterization $(\theta_{m},\gamma_{m})$ by $\xi$.
Under parameterization $\xi$, we denote the true parameter point by $\xi^{*}$.

To avoid the collision of indices,
we use index $i,j,k$ for observation $x_{i}$,
index $s,t,u$ for parameter $\omega^{s}$,
and
index $a,b,c$ for parameter $\theta_{m}^{a}$.
We use index $\kappa,\lambda,\mu$ for parameter $\gamma_{m}^{\kappa}$,
index $\alpha,\beta,\gamma$ for parameter $\xi^{\alpha}$,
and
index $m,n,l$ for submodel $\mathcal{M}_{m}$.

For simplicity,
we denote the Kullback--Leibler risk 
by $R(\omega^{*},\hat{q})$,
i.e.,
the function of the true parameter point $\omega^{*}$ and predictive distribution $\hat{q}$.
We denote the expectation with respect to the distribution with the parameter point $\omega$ 
by $\mathrm{E}_{\omega}$.

We consider two maximum likelihood estimators.
We denote
the maximum likelihood estimator of $p(x^{(N)}|\omega)$ by $\hat{\omega}(x^{(N)})$
and
the restricted maximum likelihood estimator
of $p(x^{(N)}|\omega(\theta_{m},0))$
by $\hat{\theta}_{m}(x^{(N)})$.

We consider the projection of the true parameter point into $\Theta_{m}$.
We denote
the best approximating point of $\omega^{*}$ with respect to $p_{m}(x^{(N)}|\theta_{m})$
by $\theta^{(p)}_{m}$.
In other words,
$\theta^{(p)}_{m}$ is defined by
\begin{eqnarray*}
	\theta^{(p)}_{m}=\mathop{\mathrm{argmax}}_{\theta_{m}\in\Theta_{m}}
\mathrm{E}_{\omega^{*}}[\log p(x^{(N)}|\omega(\theta_{m},0))].
\end{eqnarray*}

We denote the $(i,j)$-component of the Fisher information matrix of $p(x^{(N)}|\omega)$
by $g^{(p)}_{ij}(\omega)$
and that of $q(y^{(M)}|\omega)$
by $g^{(q)}_{ij}(\omega)$,
and we denote the $(\alpha,\beta)$-components of those
with respect to parameter $\xi$
by $g^{(p)}_{\alpha\beta}(\xi)$ and $g^{(q)}_{\alpha\beta}(\xi)$,
respectively.
We denote the $(a,b)$-component 
of the sub-matrix with respect to $\theta_{m}$
of Fisher information matrix $g^{(p)}_{\alpha\beta}(\xi)$
by $g^{(p)}_{ab}(\theta_{m})$
and that of $g^{(q)}_{\alpha\beta}(\xi)$
by $g^{(q)}_{ab}(\theta_{m})$.
We denote the sub-matrices with (a,b)-components as $g^{(p)}_{ab}(\theta_{m})$ and $g^{(q)}_{ab}(\theta_{m})$
by $g^{(p)}(\theta_{m})$ and $g^{(q)}(\theta_{m})$, respectively.

We write the upper index $-1$ to denote the inverse of the matrix;
we denote the inverses of Fisher information matrices 
$g^{(p)}(\omega)$, $g^{(q)}(\omega)$, $g^{(p)}(\xi)$, and $g^{(q)}(\xi)$
by
$g^{(p)-1}(\omega)$, 
$g^{(q)-1}(\omega)$, $g^{(p)-1}(\xi)$,
and
$g^{(q)-1}(\xi)$,
respectively.
We use the upper index for the components of the inverse of the Fisher information matrix;
we denote the $(i,j)$-components of the inverse Fisher information matrices $g^{(p)-1}$ and $g^{(q)-1}$
by $g^{(p) ij}(\omega)$ and $g^{(q)ij}(\omega)$,
respectively.
We denote the $(a,b)$-components of the inverse Fisher information matrices $g^{(p)-1}(\theta_{m})$ and $g^{(q)-1}(\theta_{m})$
by $g^{(p)ab}_{m}(\theta_{m})$ and $g^{(q)ab}_{m}(\theta_{m})$,
respectively.
Note that
the $(a,b)$-component of the inverse Fisher information matrix 
with $(\alpha,\beta)$-component as $g^{(p)\alpha\beta}(\xi(\theta_{m},0))$
is not generally identical to $g^{(p)ab}_{m}(\theta_{m})$.
We adopt Einstein summation convention:
if the same indices appear in any one term,
it implies summation over that index.

For the model selection,
we consider local misspecification.
The local misspecification is that
the true parameter point $\xi^{*}$
and
submodel $\mathcal{M}_{m}$
satisfy the following equation:
\begin{eqnarray}
\sqrt{N}\{\xi^{*\alpha}-\xi^{\alpha}(\theta^{(p)}_{m},0)\}
=h^{\alpha}
\quad\mathrm{for}\quad \alpha=1,\ldots,d_{\mathrm{full}}.
\label{Localmisspec}
\end{eqnarray}
If $h$ vanishes,
the assumption means
that the true distribution is included in submodel $\mathcal{M}_{m}$.
Thus,
the assumption is an extension of the assumption 
that the true distribution is included in submodel $\mathcal{M}_{m}$.
The assumption is known as local alternatives in statistical test theory.
See \citet{vanderVaart}.
The local misspecification in the model selection context is argued, for example,
in
\citet{Shimodaira(1997)},
\citet{HjortandClaeskens(2003)},
and \citet{ClaeskensandHjort(2003)}.
See also \citet{LeebandPotscher(2005)}.
Note that the assumption does not depend on parameterizations:
if we adopt parameterization $\omega$,
the assumption (\ref{Localmisspec}) is denoted by
\begin{eqnarray}
\sqrt{N}\{\omega^{*s}-\omega^{s}(\theta^{(p)}_{m},0)\}=\frac{\partial \omega^{s}}{\partial \xi^{\alpha}}(\xi^{*})h^{\alpha}+\mathrm{o}(1)
\quad\mathrm{for}\quad s=1,\ldots,d_{\mathrm{full}}.
\label{Localmisspec_omega}
\end{eqnarray}
In this parameterization,
we denote $\frac{\partial\omega^{s}}{\partial\xi^{\alpha}}(\xi^{*})h^{\alpha}$ in (\ref{Localmisspec_omega}) by $h^{s}$.

\section{Multi-step ahead predictions under local misspecification}

First,
we expand the Kullback--Leibler risk of the Bayesian predictive distribution
in multistep ahead predictions under local misspecification.
Next,
we show that the Kullback--Leibler risk of the Bayesian predictive distribution 
is asymptotically smaller 
than that of the plug-in predictive distribution.

\begin{thm}
\label{Risk_decomp}
Assume that the true parameter point $\xi^{*}$ and submodel $\mathcal{M}_{m}$ satisfy (\ref{Localmisspec}).
Then,
for any smooth prior $\pi$,
the Kullback--Leibler risk
of the Bayesian predictive distribution $q_{m,\pi}$ in submodel $\mathcal{M}_{m}$
is asymptotically expanded as
\begin{eqnarray}
R(\omega^{*},q_{m,\pi})&=&
\frac{1}{2N}S_{\alpha\beta}(\xi^{*})h^{\alpha}h^{\beta}
+\frac{1}{2}\log\frac{|g^{(p)}(\theta^{(p)}_{m})+g^{(q)}(\theta^{(p)}_{m})|}{|g^{(p)}(\theta^{(p)}_{m})|}+\mathrm{o}(1),
\label{KL_NND_Bayes}
\end{eqnarray}
where $|\cdot|$ is a determinant and $S_{\alpha\beta}(\xi^{*})$ is the $(\alpha,\beta)$-component of the matrix given by
\begin{eqnarray*}
	S(\xi^{*})&=&
\left(g^{(q)-1}(\xi^{*})
+\begin{pmatrix}g_{m}^{(p)-1}(\theta^{(p)}_{m}) & \bigzerou^{\top}_{(d_{\mathrm{full}}-d_{m})\times d_{m}} \\ \bigzerol_{(d_{\mathrm{full}}-d_{m})\times d_{m}} & \bigzerol_{(d_{\mathrm{full}}-d_{m})\times (d_{\mathrm{full}}-d_{m})} \end{pmatrix}\right)^{-1}.
\end{eqnarray*}
Here, $\bigzerol_{(d_{\mathrm{full}}-d_{m})\times d_{m}}$ is the $(d_{\mathrm{full}}-d_{m})\times d_{m}$-dimensional zero matrix
and
$\bigzerol_{(d_{\mathrm{full}}-d_{m})\times (d_{\mathrm{full}}-d_{m})}$ is the $(d_{\mathrm{full}}-d_{m})\times (d_{\mathrm{full}}-d_{m})$-dimensional zero matrix.
\end{thm}
The proof is given in the appendix.
The expansion is invariant up to constant order under the reparameterization $\omega$ in the full model.
See (\ref{L_exp_final}) in the appendix.

\begin{rem}
Note that the asymptotic Kullback--Leibler risk of the Bayesian predictive distribution 
does not depend on priors up to constant order.
This corresponds to the fact that
the asymptotic Kullback--Leibler risk of the Bayesian predictive distribution in one-step ahead predictions
does not depend on priors up to the $N^{-1}$ order.
If $h$ vanishes and if the data and the target variables are identically and identically distributed,
then,
$R(\omega^{*},q_{m,\pi})$ is given by $d\log\{(N+M)/N\}/2$ up to constant order.
In one-step ahead predictions,
it is known that the asymptotic Kullback--Leibler risk of the Bayesian predictive distributions is given as $d/(2N)$ up to the $N^{-1}$ order.
The Bayesian predictive distribution $q_{m,\pi}$ is decomposed as
\begin{eqnarray*}
q_{m,\pi}(y^{(M)}|x^{(N)})&=&q_{m,\pi}(y_{M}|x^{(N)},y^{(M-1)})q_{m,\pi}(y_{M-1}|x^{(N)},y^{(M-2)})\ldots q_{m,\pi}(y_{1}|x^{(N)}).
\end{eqnarray*}
Since the Kullback--Leibler risk of the Bayesian predictive distribution is decomposed according to the above decomposition,
$R(\omega^{*},q_{m,\pi})$ is also calculated as $\lim_{N\rightarrow\infty}\Sigma_{j=1}^{M}d/(2N+2j)$.
This is equal to $d\log\{(N+M)/N\}/2$.
\end{rem}

By using the above theorem,
we show that the Bayesian predictive distribution has smaller Kullback--Leibler risk
than the plug-in predictive distribution in the multistep ahead prediction.

\begin{thm}
Assume that the true parameter point $\xi^{*}$ and submodel $\mathcal{M}_{m}$ satisfy (\ref{Localmisspec}).
Then,
for any smooth prior $\pi$,
the Kullback--Leibler risk $R(\omega^{*},q_{m,\pi})$ 
of the Bayesian predictive distribution 
in submodel $\mathcal{M}_{m}$
is smaller in constant order
than the Kullback--Lebler risk $R(\omega^{*},q_{m}(\cdot|\hat{\theta}_{m}))$ 
of the plug-in predictive distribution 
with the maximum likelihood estimator in submodel $\mathcal{M}_{m}$:
\begin{eqnarray*}
\lim_{N\rightarrow\infty} R(\omega^{*},q_{m,\pi})\geq \lim_{N\rightarrow\infty} R(\omega^{*},q_{m}(\cdot|\hat{\theta}_{m})).
\end{eqnarray*}

\end{thm}
\begin{proof}
From the Taylor expansion
and
from (\ref{MLE_linear_imb_p}) in the appendix,
the Kullback--Leibler risk $R(\omega^{*},q_{m}(\cdot|\hat{\theta}_{m}))$ is expanded as
\begin{eqnarray*}
R(\omega^{*},q_{m}(\cdot|\hat{\theta}_{m}))
&=&
\frac{1}{2}g^{(q)}_{st}(\omega^{*})\mathrm{E}_{\omega^{*}}
[\{\omega^{*s}-\omega^{s}(\hat{\theta}_{m}(x^{(N)}),0)\}
\{\omega^{*t}-\omega^{t}(\hat{\theta}_{m}(x^{(N)}),0)\}]+\mathrm{o}(1)
\nonumber\\
&=&\frac{1}{2N}g^{(q)}_{\alpha\beta}(\xi^{*})h^{\alpha}h^{\beta}
+\frac{1}{2}g^{(q)ab}_{m}(\theta^{(p)}_{m})g^{(p)}_{ab}(\theta^{(p)}_{m})+\mathrm{o}(1).
\end{eqnarray*}
Since the Fisher information matrices $g^{(p)}(\theta^{(p)}_{m})$ and $g^{(q)}(\theta^{(p)}_{m})$ are positive semidefinite,
the following inequality holds:
\begin{eqnarray*}
\log\frac{|g^{(p)}(\theta^{(p)}_{m})+g^{(q)}(\theta^{(p)}_{m})|}{|g^{(p)}(\theta^{(p)}_{m})|}\geq g^{(p)ab}_{m}(\theta^{(p)}_{m})g^{(q)}_{ab}(\theta^{(p)}_{m}).
\end{eqnarray*}
From the inequality that $g^{(q)}(\xi^{*})\succeq S$,
we have
\begin{eqnarray*}
g^{(q)}_{\alpha\beta}(\xi^{*})h^{\alpha}h^{\beta} \geq S_{\alpha\beta}h^{\alpha}h^{\beta},
\end{eqnarray*}
where the binary relation $A \succeq B$ means that $A-B$ is positive semidefinite.
Thus, we complete the proof.
\end{proof}

\begin{rem}
This theorem implies that we should use the Bayesian predictive distribution for multistep ahead predictions
instead of the plug-in predictive distribution from the viewpoint of Kullback--Leibler risk.
Thus,
we consider the information criteria when we use the Bayesian predictive distribution in the selected model.
In one-step ahead prediction,
it is well-known that 
the Bayesian predictive distribution has smaller Kullback--Leibler risk
than the plug-in predictive distribution up to the $N^{-2}$ order.
See \citet{Komaki(1996)}, \citet{Hartigan(1998)}, and \citet{Komaki(2014)}.
\citet{KonishiandKitagawa(2003)} construct information criteria
when using the Bayesian predictive distribution in one-step ahead predictions.
\end{rem}

\begin{rem}
The result is related to the prediction in the locally asymptotically mixed normal (LAMN) models as follows:
due to the LAMN property,
we consider the prediction of the target variables based on the data
conditioning on the two Fisher information matrices of the data and the target variables.
In our setting,
we also consider the prediction of the target variables based on the data
conditioning on the two Fisher information matrices of the data and the target variables.
Indeed, the Kullback--Leibler risk 
of the Bayesian predictive distributions (\ref{KL_NND_Bayes}) 
has the same form as (2) in \citet{SeiandKomaki(2007)}.
\end{rem}

\section{Information criteria for multistep ahead predictions}

On the basis of the results in the previous section,
we construct an information criterion by using an asymptotically unbiased estimator of the Kullback--Leibler risk.

\begin{thm}
\label{Unbiased}
Let $\hat{R}(m)$ be an estimator of the Kullback--Leibler risk of the Bayesian predictive distribution
in submodel $\mathcal{M}_{m}$ given by
\begin{eqnarray}
\hat{R}(m)&=&
\frac{1}{2N}
\hat{S}_{\alpha\beta}\hat{h}^{\alpha}\hat{h}^{\beta}
+\frac{1}{2}\hat{S}_{ab}g^{(p)ab}_{m}(\hat{\theta}_{m})
-\frac{1}{2}\hat{S}_{\alpha\beta}g^{(p)\alpha\beta}(\hat{\xi})
\nonumber\\
&&+\frac{1}{2}\log\frac{|g^{(p)}(\hat{\theta}_{m})+g^{(q)}(\hat{\theta_{m}})|}{|g^{(p)}(\hat{\theta}_{m})|},
\end{eqnarray}
where $\hat{S}_{\alpha\beta}$ is the $(\alpha,\beta)$-component of the matrix given by
\begin{eqnarray*}
\hat{S}&=&\left(g^{(q)-1}(\hat{\xi})+\begin{pmatrix}g_{m}^{(p)-1}(\hat{\theta}_{m}) & \bigzerou^{\top}_{(d_{\mathrm{full}}-d_{m})\times d_{m}} \\
 \bigzerol_{(d_{\mathrm{full}}-d_{m})\times d_{m}} & \bigzerol_{(d_{\mathrm{full}}-d_{m})\times (d_{\mathrm{full}}-d_{m})} \end{pmatrix}\right)^{-1}
\end{eqnarray*}
and for $\alpha\in\{1,\ldots,d_{\mathrm{full}}\}$,
$\hat{h}^{\alpha}$ is given by $\hat{h}^{\alpha}/\sqrt{N}=\hat{\xi}^{\alpha}-\xi^{\alpha}(\hat{\theta}_{m},0)$.
Assume that the true parameter point $\xi^{*}$ and submodel $\mathcal{M}_{m}$ satisfy (\ref{Localmisspec}).
Then, $\hat{R}(m)$ is an asymptotically unbiased estimator 
of the Kullback--Leibler risk $R(\omega^{*},q_{m,\pi})$.
\end{thm}

The proof is given in the appendix.

From Theorem \ref{Unbiased},
we propose the following model selection criterion as the multistep predictive information criterion ($\mathrm{MSPIC}$):
\begin{eqnarray*}
\mathrm{MSPIC}(m)&=&2\hat{R}(m)\nonumber\\
&=&\frac{1}{N}\hat{S}_{\alpha\beta}\hat{h}^{\alpha}\hat{h}^{\beta}
+\hat{S}_{ab}g^{(p)ab}_{m}(\hat{\theta}_{m})
-\hat{S}_{\alpha\beta}g^{(p)\alpha\beta}(\hat{\xi})
+\log\frac{|g^{(p)}(\hat{\theta}_{m})+g^{(q)}(\hat{\theta}_{m})|}{|g^{(p)}(\hat{\theta}_{m})|}.
\end{eqnarray*}
Here, we multiply $\hat{R}(m)$ by 2 to make the definition consistent with $\mathrm{AIC}$ (\citealp{Akaike(1973)}).
If two Fisher information matrices $g^{(p)}(\theta_{m})$ and $g^{(q)}(\theta_{m})$ are identical,
$\mathrm{MSPIC}$ coincides with PIC (\citealp{Kitagawa(1997)})
when using the uniform prior and with predictive likelihood (\citealp{Akaike(1980)}).

We also consider the bootstrap adjustment of $\mathrm{MSPIC}$.
First,
we generate $B$ bootstrap samples $x^{(N)}_{1},\ldots,x^{(N)}_{b},\ldots,x^{(N)}_{B}$ 
via a parametric or non-parametric bootstrap method using the full model.
Second, for each $b$ in $\{1,\ldots,B\}$, we calculate
the value of $\mathrm{MSPIC}_{1}(m;x^{(N)}_{b})$
where
$\mathrm{MSPIC}_{1}(m;x^{(N)}_{b})$ is the value of
\begin{eqnarray*}
\frac{1}{N}\hat{S}_{\alpha\beta}\hat{h}^{\alpha}\hat{h}^{\beta}
+\hat{S}_{ab}g^{(p)ab}_{m}(\hat{\theta}_{m})
-\hat{S}_{\alpha\beta}g^{(p)\alpha\beta}(\hat{\xi})
\end{eqnarray*}
using $x^{(N)}_{b}$ instead of $x^{(N)}$.
Finally,
we obtain
\begin{eqnarray*}
	\mathrm{MSPIC}_{\mathrm{BS}}(m)
	&=&
	\frac{1}{B}\Sigma_{b=1}^{B}\mathrm{MSPIC}_{1}(m;x^{(N)}_{b})+\log\frac{|g^{(p)}(\hat{\theta}_{m})+g^{(q)}(\hat{\theta}_{m})|}{|g^{(p)}(\hat{\theta}_{m})|}.
\end{eqnarray*}
Consider the first three terms in the definition of $\mathrm{MSPIC}$.
These terms are an asymptotically unbiased estimator of $S_{\alpha\beta}h^{\alpha}h^{\beta}/N$.
However,
this estimator may have excessive variance 
because the matrix $\hat{S}$ is not equal to the asymptotic variance of $\hat{h}^{\alpha}$.
To avoid the excessive variance of the estimator, we use the bootstrap method.
\citet{LvandLiu(2014)} applied the bootstrap adjustment of $\mathrm{TIC}$ (\citealp{Takeuchi(1976)}).

\section{Numerical experiments}

We show that the proposed information criteria are effective for the multistep ahead prediction
through two numerical experiments.
After the model selections by $\mathrm{AIC}$, $\mathrm{PIC}$, $\mathrm{MSPIC}$,
and its bootstrap adjustment $\mathrm{MSPIC}_{\mathrm{BS}}$,
we evaluate the predictive performance of the selected models as follows:
the derivation of $\mathrm{AIC}$ is based on the plug-in predictive distribution with the maximum likelihood.
In contrast,
those of $\mathrm{PIC}$, $\mathrm{MSPIC}$, and $\mathrm{MSPIC}_{\mathrm{BS}}$ 
are based on the Bayesian predictive distribution.
Thus,
the predictive performance of the $\mathrm{AIC}$-best model
is evaluated by the goodness of the plug-in predictive distribution $q_{m}(\cdot|\hat{\theta}_{m})$
in the $\mathrm{AIC}$-best model.
In contrast,
the predictive performance of the $\mathrm{PIC}$-best, the $\mathrm{MSPIC}$-best, and the $\mathrm{MSPIC}_{\mathrm{BS}}$-best models
is evaluated by the goodness of the Bayesian predictive distributions $q_{m,\pi}(\cdot|\cdot)$
in the $\mathrm{PIC}$-best, the $\mathrm{MSPIC}$-best, and the $\mathrm{MSPIC}_{\mathrm{BS}}$-best models.

We consider the empirical goodness of the predictive distribution as follows.
We generate the data and the target variables $R$ times
and calculate the mean of minus $\log$ predictive densities 
$-\mathop{\Sigma}_{r=1}^{R}\log \hat{q}(y^{(M)}_{r}|x^{(N)}_{r})$
of each information criterion.
Here, for $r=1,\ldots,R$,
$x^{(N)}_{r}$ and $y^{(M)}_{r}$ are the $r$-th data and the $r$-th target variables.
It is preferable 
that the value $-\mathop{\Sigma}_{r=1}^{R}\log \hat{q}(y^{(M)}_{r}|x^{(N)}_{r})$ is small
because it is an estimator of the Kullback--Leibler risk up to the term related to the predictive distribution.
We set $R=100$ in the first numerical experiment
and $R=10$ in the second numerical experiment.

\subsection{The extrapolation in the curve fitting}
First, consider the extrapolation in the curve fitting in the introduction.
For $m\in\{1,\ldots,d_{\mathrm{full}}\}$,
the data and the target variables in the $m$-th model are given by
\begin{eqnarray*}
	x^{(N) \top}=\Phi_{m}\theta_{m}+\epsilon_{N\times N}
	&\mathrm{and}&
	y^{(M) \top}=\tilde{\Phi}_{m}\theta_{m}+\tilde{\epsilon}_{M\times M},
\end{eqnarray*}
where $\Phi_{m}$ and $\tilde{\Phi}_{m}$ are design matrices defined by
\begin{eqnarray*}
\Phi_{m}=
\begin{pmatrix}
\phi_{1}(z_{1}) & \ldots & \phi_{d_{m}}(z_{1}) \\
\ldots          & \ldots & \ldots \\
\phi_{1}(z_{N}) & \ldots & \phi_{d_{m}}(z_{N})
\end{pmatrix}
&\mathrm{and}&
\tilde{\Phi}_{m}=
\begin{pmatrix}
\phi_{1}(z_{N+1}) & \ldots & \phi_{d_{m}}(z_{N+1}) \\
\ldots          & \ldots & \ldots \\
\phi_{1}(z_{N+M}) & \ldots & \phi_{d_{m}}(z_{N+M})
\end{pmatrix},
\end{eqnarray*}
respectively.
For simplicity, we denote $\Phi_{d_{\mathrm{full}}}$, $\tilde{\Phi}_{d_\mathrm{full}}$, and $\theta_{d_{\mathrm{full}}}$ by $\Phi$, $\tilde{\Phi}$, and $\theta$,
respectively.
We denote the maximum likelihood estimator of $\theta$ by $\hat{\theta}$.

The information criteria $\mathrm{AIC}$, $\mathrm{PIC}$, and $\mathrm{MSPIC}$ are given by
\begin{eqnarray}
\mathrm{AIC}(m)
&=&
(\hat{\theta}-\begin{pmatrix}\hat{\theta}_{m} \\ 0\end{pmatrix})^{\top}
S_{\mathrm{AIC}}
(\hat{\theta}-\begin{pmatrix}\hat{\theta}_{m} \\ 0\end{pmatrix})
+2d_{m}-d_{\mathrm{full}},
\end{eqnarray}

\begin{eqnarray}
\mathrm{PIC}(m)&=&
(\hat{\theta}-\begin{pmatrix}\hat{\theta}_{m} \\ 0 \end{pmatrix})^{\top}
S_{\mathrm{PIC}}
(\hat{\theta}-\begin{pmatrix}\hat{\theta}_{m} \\ 0 \end{pmatrix})
	+d_{m}\log 2+d_{m}-d_{\mathrm{full}},
\end{eqnarray}
and
\begin{eqnarray}
\mathrm{MSPIC}(m)&=&
(\hat{\theta}-\begin{pmatrix}\hat{\theta}_{m} \\ 0 \end{pmatrix})^{\top}
S_{\mathrm{MSPIC}}
(\hat{\theta}-\begin{pmatrix}\hat{\theta}_{m} \\ 0 \end{pmatrix})
+\log\frac{|\Phi_{m}^{\top}\Phi_{m}+\tilde{\Phi}_{m}^{\top}\tilde{\Phi}_{m}|}{|\Phi_{m}^{\top}\Phi_{m}|}
\nonumber\\
&&+\mathrm{tr}\begin{pmatrix} \sigma^{2} (\Phi_{m}^{\top}\Phi_{m})^{-1} & \bigzerou^{\top}_{(d_{\mathrm{full}}-d_{m})\times d_{m}} 
\\ \bigzerol_{(d_{\mathrm{full}}-d_{m})\times d_{m}} & \bigzerol_{(d_{\mathrm{full}}-d_{m})\times (d_{\mathrm{full}}-d_{m})} \end{pmatrix}S_{\mathrm{MSPIC}}
\nonumber\\
&&-\mathrm{tr}(\Phi^{\top}\Phi)^{-1}S_{\mathrm{MSPIC}},
\end{eqnarray}
where $S_{\mathrm{AIC}}$, $S_{\mathrm{PIC}}$, and $S_{\mathrm{MSPIC}}$ are given by
\begin{eqnarray}
S_{\mathrm{AIC}}&=&\frac{1}{\sigma^{2}}\Phi^{\top}\Phi,
\end{eqnarray}
\begin{eqnarray}
S_{\mathrm{PIC}}&=&\frac{1}{\sigma^{2}}\left((\Phi^{\top}\Phi)^{-1}
+\begin{pmatrix}(\Phi_{m}^{\top}\Phi_{m})^{-1} & \bigzerou^{\top}_{(d_{\mathrm{full}}-d_{m})\times d_{m}}  \\ \bigzerol_{(d_{\mathrm{full}}-d_{m})\times d_{m}} & \bigzerol_{(d_{\mathrm{full}}-d_{m})\times (d_{\mathrm{full}}-d_{m})}  \end{pmatrix}\right)^{-1},
\end{eqnarray}
and
\begin{eqnarray}
S_{\mathrm{MSPIC}}&=&\frac{1}{\sigma^{2}}\left(
(\tilde{\Phi}^{\top}\tilde{\Phi})^{-1}+\begin{pmatrix}(\Phi_{m}^{\top}\Phi_{m})^{-1} & \bigzerou^{\top}_{(d_{\mathrm{full}}-d_{m})\times d_{m}}  
\\ \bigzerol_{(d_{\mathrm{full}}-d_{m})\times d_{m}}  & \bigzerol_{(d_{\mathrm{full}}-d_{m})\times (d_{\mathrm{full}}-d_{m})} \end{pmatrix}\right)^{-1},
\end{eqnarray}
respectively.

As the sets of functions $\{\phi_{a}\}_{a=1}^{d_{\mathrm{full}}}$,
we use
trigonometric functions $\{\phi_{\mathrm{tri},a}\}_{a=1}^{d_{\mathrm{full}}}$:
\begin{eqnarray*}
\phi_{\mathrm{tri},a}(z)&=&
\begin{cases}
1 & (a=1), \\
\sqrt{2}\cos(2\pi\frac{a}{2}z) & (a:\mathrm{even}),\\
\sqrt{2}\sin(2\pi\frac{a-1}{2}z) & (a:\mathrm{odd}).
\end{cases}
\end{eqnarray*}
For all $i\in\{1,\ldots,N+M\}$,
we design $z_{i}$ as $\alpha \times(i/N)$ where $\alpha$ is in [0,1].

We generate the data and the target variables as follows:
\begin{eqnarray*}
x^{(N) \top}=\begin{pmatrix}f(z_{1})\\ f(z_{2}) \\ \ldots \\ f(z_{N}) \end{pmatrix}
+\epsilon_{N\times N}
&\mathrm{and}&
y^{(M) \top}=\begin{pmatrix}f(z_{N+1})\\ f(z_{N+2}) \\ \ldots \\ f(z_{M}) \end{pmatrix}
+\tilde{\epsilon}_{M\times M}.
\end{eqnarray*}
In this experiment,
we compare 
the minus log plug-in predictive distribution 
with the maximum likelihood estimator in the $\mathrm{AIC}$-best model
and
the minus log Bayesian predictive distribution with the uniform prior
given by
\begin{eqnarray*}
-\log q_{m,\pi}(y^{(M)}|x^{(N)})&=&
\frac{1}{2\sigma^{2}}
\left\vert
\begin{pmatrix}x^{(N)\top} \\ y^{(M)\top}\end{pmatrix}
-\begin{pmatrix}\Phi_{m}\\ \tilde{\Phi}_{m}\end{pmatrix}\hat{\theta}_{m}(x^{(N)},y^{(M)})
\right\vert^{2}
-\frac{1}{2\sigma^{2}}
\left\vert
x^{(N)\top}-\Phi_{m}\hat{\theta}_{m}(x^{(N)})
\right\vert^{2}
\nonumber\\
&&+\frac{M}{2}\log(2\pi\sigma^{2})
+\frac{1}{2}\log\frac{\vert\Phi_{m}^{\top}\Phi_{m}+\tilde{\Phi}_{m}^{\top}\tilde{\Phi}_{m}\vert}{\vert\Phi_{m}^{\top}\Phi_{m}\vert}
\end{eqnarray*}
of the $\mathrm{PIC}$-best, the $\mathrm{MSPIC}$-best, and the $\mathrm{MSPIC}_{\mathrm{BS}}$-best models.
Here, we denote the maximum likelihood estimator of $r_{m}(x^{(N)},y^{(M)}|\theta_{m})$ by $\hat{\theta}_{m}$.

\begin{table}[t]
\caption{The mean of the minus $\log$ predictive densities 
when the true function is $f_{1}$ and $\alpha$ is 1. The lowest value in each row is underlined.}
\centering
\begin{tabular}{r|rrrr}
  \hline
$N$ and $M$ & $\mathrm{AIC}$ & $\mathrm{PIC}$ & $\mathrm{MSPIC}$ & $\mathrm{MSPIC}_{\mathrm{BS}}$ \\ 
  \hline
100 and 100 & $-4.43$ & $-8.71$ & $-8.71$ & \underline{$-9.11$} \\ 
100 and 200 & $-9.52$ & $-21.84$ & $-22.20$ & \underline{$-23.04$} \\ 
100 and 500 & $-19.26$ & $-62.33$ & $-65.96$ & \underline{$-67.51$} \\ 
100 and 1000 & $-40.93$ & $-139.66$ & $-150.30$ & \underline{$-152.55$} \\ 
   \hline
\end{tabular}
\label{Table:trigonoseries:alpha1}
\end{table}

\begin{table}[t]
\caption{The mean of the minus $\log$ predictive densities 
when the true function is $f_{2}$ and $\alpha$ is 1. The lowest value in each row is underlined.}
\centering
\begin{tabular}{r|rrrr}
  \hline
$N$ and $M$ & $\mathrm{AIC}$ & $\mathrm{PIC}$ & $\mathrm{MSPIC}$ & $\mathrm{MSPIC}_{\mathrm{BS}}$ \\ 
  \hline
100 and 100 & $-11.44$ & $-13.28$ & $-13.28$ & \underline{$-13.57$} \\ 
100 and 200 & $-21.53$ & $-28.08$ & $-28.32$ & \underline{$-28.58$} \\ 
100 and 500 & $-60.21$ & $-79.27$ & $-79.94$ & \underline{$-81.73$} \\ 
100 and 1000 & $-116.74$ & $-158.14$ & $-161.33$ & \underline{$-165.81$} \\ 
   \hline
\end{tabular}
\label{Table:infiniteseries:alpha1}
\end{table}

\begin{table}[t]
\caption{The mean of the minus $\log$ predictive densities when the true function is $f_{2}$ and $\alpha$ is 0.9. The lowest value in each row is underlined.}
\centering
\begin{tabular}{r|rrrr}
  \hline
$N$ and $M$ & $\mathrm{AIC}$ & $\mathrm{PIC}$ & $\mathrm{MSPIC}$ & $\mathrm{MSPIC}_{\mathrm{BS}}$ \\ 
  \hline
100 and 100 & $-8.91$ & \underline{$-13.48$} & $-12.98$ & $-13.23$ \\
100 and 200 & $-14.88$ & $-27.08$ & $-26.98$ & \underline{$-27.38$} \\
100 and 500 & $-20.99$ & $-68.47$ & $-70.99$ & \underline{$-72.98$} \\
100 and 1000 & $-75.39$ & $-154.72$ & $-158.20$ & \underline{$-163.44$} \\
   \hline
\end{tabular}
\label{Table:infiniteseries:alpha09}
\end{table}

First, we consider the setting where the true function $f_{1}$ is given by
\begin{eqnarray*}
f_{1}(z)&=&2\sin(2\pi\times z)+0.2\sin(2\pi\times 4z)\\
	    &&+0.1\sin(2\pi\times 8z)+0.1\sin(2\pi\times 12z),
\end{eqnarray*}
where $\sigma^{2}=(0.2)^2$ and $\alpha=1.0$.
We let $d_{\mathrm{full}}=31$.
Table \ref{Table:trigonoseries:alpha1} shows that 
$\mathrm{MSPIC}_{\mathrm{BS}}$ has the lowest value,
regardless of $N$ and $M$ when $\alpha$ is 1.

Second, we consider the setting where the true function $f_{2}$ is given by
\begin{eqnarray*}
f_{2}(z)=\frac{\pi^2}{6}-\frac{\pi}{2}(z\,\mathrm{mod}\,2\pi)+\frac{1}{4}(z\,\mathrm{mod}\,2\pi)^2.
\end{eqnarray*}
We set $\sigma^2=(0.2)^2$ and $d_{\mathrm{full}}=16$.
We consider the settings with $\alpha=1$ and $\alpha=0.9$.
Table \ref{Table:infiniteseries:alpha1} shows
that 
when $\alpha$ is 1,
$\mathrm{MSPIC}_{\mathrm{BS}}$ has the lowest value of the minus log predictive distribution,
regardless of the ratio of $N$ and $M$.
Table \ref{Table:infiniteseries:alpha09} shows that
when $\alpha$ is 0.9,
$\mathrm{MSPIC}_{\mathrm{BS}}$ has the lowest value
except when $N$ and $M$ are 100 and 100, respectively.

There is difference between the first and second settings.
In the first setting,
the true function $f_{1}$ is included in the full model.
In the second setting,
the true function $f_{2}$ is not included in the full model.
See \citet{Shibata(1981)} for details related to the second setting.
However,
the experiments indicate 
that $\mathrm{MSPIC}_{\mathrm{BS}}$ works well in both settings
and
that the dominance of $\mathrm{MSPIC}_{\mathrm{BS}}$ is enlarged as the ratio of $N$ and $M$ grows.

\subsection{Normal regression model with an unknown variance}

Next, consider the normal regression model with an unknown variance.
We consider the full model given by
\begin{eqnarray*}
x^{(N)\top}=\Phi\theta+\sigma\epsilon_{N\times N}
&\mathrm{and}&
y^{(M)\top}=\tilde{\Phi}\theta+\sigma\tilde{\epsilon}_{M\times M},
\end{eqnarray*}
respectively.
Here, $\Phi$ and $\tilde{\Phi}$ are $N\times 10$ and $M\times 10$ design matrices, respectively.
The parameters $\theta$ and $\sigma$ are unknown.
We consider 511 submodels given by the models with the restriction that some components of $\theta$ vanish.
We denote the design matrix in the $m$-th model by $\Phi_{m}$ 
and denote the $m$-th model $\mathcal{M}_{m}$ by
\begin{eqnarray*}
x^{(N)\top}=\Phi_{m}\theta_{m}+\sigma\epsilon_{N\times N}
&\mathrm{and}&
y^{(M)\top}=\tilde{\Phi}_{m}\theta_{m}+\sigma\tilde{\epsilon}_{M\times M},
\end{eqnarray*}
respectively.

\begin{table}[t]
\caption{The mean of the minus $\log$ predictive densities in the setting 
where the parameter $\lambda$ is 1, 10, 50, and 100
and
the sample sizes $N$ and $M$ are 50 and 250, respectively.
The lowest value in each row is underlined.}
\centering
\begin{tabular}{r|rrrr}
  \hline
$\lambda$ & $\mathrm{AIC}$ & $\mathrm{PIC}$ & $\mathrm{MSPIC}$ & $\mathrm{MSPIC}_{BS}$ \\ 
  \hline
  1 & $-176.77$ & $-201.92$ & $-202.07$ & \underline{$-205.40$} \\ 
  10 & $-126.97$ & \underline{$-211.60$} & $22.55$ & $-209.33$ \\ 
  50 & $1176.34$ & $-180.16$ & $544.08$ & \underline{$-188.78$} \\ 
  100 & $5496.64$ & $-75.54$ & $750.14$ & \underline{$-180.80$} \\ 
  150 & $14922.99$ & $-75.41$ & $871.94$ & \underline{$-178.16$} \\ 
  200 & $33812.08$ & $38.62$ & $957.92$ & \underline{$-182.71$} \\ 
   \hline
\end{tabular}
\label{Table:varianceunknown_lambdachange:N=50}
\end{table}

\begin{table}[t]
\caption{The mean of the minus $\log$ predictive densities in the setting 
where the parameter $\lambda$ is 1, 10, 50, and 100
and
the sample sizes $N$ and $M$ are 100 and 500, respectively.
The lowest value in each row is underlined.}
\centering
\begin{tabular}{r|rrrr}
  \hline
$\lambda$ & $\mathrm{AIC}$ & $\mathrm{PIC}$ & $\mathrm{MSPIC}$ & $\mathrm{MSPIC}_{BS}$ \\ 
  \hline
  1 & $-418.87$ & \underline{$-438.15$} & \underline{$-438.15$} & $-436.18$ \\ 
  10 & $-361.78$ & $-418.92$ & \underline{$-419.22$} & $-416.92$ \\ 
  50 & $124.53$ & $-408.42$ & $-408.42$ & \underline{$-420.98$} \\ 
  100 & $2273.35$ & $-340.89$ & $1287.12$ & \underline{$-405.19$} \\ 
  150 & $4437.98$ & $-285.04$ & $1528.31$ & \underline{$-392.05$} \\ 
  200 & $9491.38$ & $-191.91$ & $1698.95$ & \underline{$-406.93$} \\
   \hline
\end{tabular}
\label{Table:varianceunknown_lambdachange:N=100}
\end{table}

We set $N=50$ and $M=250$.
In this setting,
we generate the full design matrices given by
\begin{eqnarray*}
\Phi=\Phi_{r} &\,\mathrm{and}\,&
 \tilde{\Phi}=\begin{pmatrix}\Phi_{r} \\ \Phi_{r} \\ \ldots \\ \Phi_{r}\end{pmatrix}
+\lambda \begin{pmatrix}I_{10\times 10} \\ \bigzerol_{(M-10) \times 10} \end{pmatrix},
\end{eqnarray*}
where $\Phi_{r}$ is given randomly and $\lambda$ is the parameter.
Here, $I_{10\times 10}$ is the $10\times 10$ identity matrix 
and
$\bigzerol_{(M-10)\times 10}$ is the $(M-10)\times 10$ zero matrix.

We compare the minus log plug-in predictive distribution given by
\begin{eqnarray*}
-\log q_{m}(y^{(M)}|\hat{\theta}_{m}(x^{(N)}))&=&\frac{M}{2}\log(2\pi)
+\frac{M}{2}\log(\vert x^{(N)\top}-\Phi_{m}(\Phi_{m}^{\top}\Phi_{m})^{-1}\Phi_{m}^{\top}x^{(N)\top} \vert^{2}/N)
\nonumber\\
&&+\frac{1}{2}\frac{\vert y^{(M)\top}-\tilde{\Phi}_{m}(\Phi_{m}^{\top}\Phi_{m})^{-1}\Phi_{m}^{\top}x^{(N)} \vert^{2}}{\vert
x^{(N)\top}-\Phi_{m}(\Phi_{m}^{\top}\Phi_{m})^{-1}\Phi_{m}^{\top}x^{(N)}\vert^{2}/N}
\end{eqnarray*}
of the $\mathrm{AIC}$-best model
and
the minus log Bayesian predictive distribution with $\pi(\theta_{m},\sigma)=1/\sigma$
given by
\begin{eqnarray*}
-\log q_{m,\pi}(y^{(M)}|x^{(N)})&=&\frac{N+M-d_{m}}{2}
\log\left(\left\vert\begin{pmatrix}x^{(N)\top} \\ y^{(M)\top} \end{pmatrix}-\begin{pmatrix}\Phi_{m} \\ \tilde{\Phi}_{m}\end{pmatrix} \hat{\theta}_{m}(x^{(N)},y^{(M)})\right\vert^{2}\right)
\nonumber\\
&&-\frac{N-d_{m}}{2}
\log\left(\left\vert x^{(N)\top}-\Phi_{m}\hat{\theta}_{m}(x^{(N)})\right\vert^{2}\right)
\nonumber\\
&&+\frac{1}{2}\log\frac{\vert\Phi_{m}^{\top}\Phi_{m}+\tilde{\Phi}_{m}^{\top}\tilde{\Phi}_{m}\vert}{\vert\Phi_{m}^{\top}\Phi_{m}\vert}
-\log \frac{\Gamma(\frac{M+N-d_{m}}{2})}{\Gamma(\frac{N-d_{m}}{2})}
\end{eqnarray*}
of the $\mathrm{PIC}$-best, the $\mathrm{MSPIC}$-best, and the $\mathrm{MSPIC}_{\mathrm{BS}}$-best models.
The choice of the prior distribution is asymptotically irrelevant according to Theorem \ref{Risk_decomp}.
The reason why we use the above Bayesian distribution is
because it is mini-max under the Kullback--Leibler risk.
See \citet{LiangandBarron(2004)}.
Tables \ref{Table:varianceunknown_lambdachange:N=50} and \ref{Table:varianceunknown_lambdachange:N=100}
show that $\mathrm{MSPIC}_{\mathrm{BS}}$ has the lowest value 
of the minus log predictive distribution, except for the setting where $\lambda$ is 10.
The dominance of $\mathrm{MSPIC}_{\mathrm{BS}}$ is enlarged 
depending on the degree of the extrapolation,
i.e.,
the value of $\lambda$.

\section{Discussion and Conclusion}
In this paper,
we have considered the multistep ahead prediction under local misspecification.
We have shown that the Bayesian predictive distribution has smaller Kullback--Leibler risk 
in the setting
than the plug-in predictive distribution,
regardless of the prior choice.
From the results,
we have proposed the information criterion $\mathrm{MSPIC}$ for the multistep ahead prediction.
The proposed information criterion $\mathrm{MSPIC}$ is an asymptotically unbiased estimator 
of the Kullback--Leibler risk of the Bayesian predictive distribution.
By considering the variance of the information criterion $\mathrm{MSPIC}$,
we have proposed the bootstrap adjustment $\mathrm{MSPIC}_{\mathrm{BS}}$.
Numerical experiments show that our proposed information criterion is effective.

\appendix
\def\thesection{Appendix} 
\renewcommand{\thelem}{\Alph{section}\arabic{lem}}
\section{}

In this appendix,
we provide proofs of Theorems \ref{Risk_decomp} and \ref{Unbiased}.
The proofs consist of three parts:
the connection formula of the best approximating points (Lemma \ref{Connectionformula}),
the expansions of the maximum likelihood estimators (Lemma \ref{MLEs}),
and
the expansions of the Kullback--Leibler risk $R(\omega^{*},q_{m,\pi})$.

We need some additional notations for the proofs.
In the appendix,
we write $\theta$ instead of $\theta_{m}$
because we fix the submodel $\mathcal{M}_{m}$
and make expansions easier to see.
The simultaneous distribution of $(x^{(N)},y^{(M)})$ is denoted by $r(x^{(N)},y^{(M)}|\omega^{*})$.
In our setting,
distribution $r(x^{(N)},y^{(M)}|\omega^{*})$ is given as 
the product $p(x^{(N)}|\omega^{*})q(y^{(M)}|\omega^{*})$.
We use notations $g^{(r)}(\omega)$ and $g^{(r)}(\theta)$ 
for the Fisher information matrices of $r(x^{(N)},y^{(M)}|\omega)$ and $r(x^{(N)},y^{(M)}|\omega(\theta,0))$,
respectively.
Note that $g^{(r)}(\omega)=g^{(p)}(\omega)+g^{(q)}(\omega)$.
We denote $g^{(p)}_{a\alpha}\frac{\partial\xi^{\alpha}}{\partial\omega^{s}}$ by $g^{(p)}_{as}$
and use $g^{(r)}_{as}$ and $g^{(q)}_{as}$ in the same manner.

We denote the maximum likelihood estimator of $r(x^{(N)},y^{(M)}|\omega)$ by $\hat{\omega}(x^{(N)},y^{(M)})$
and the restricted maximum likelihood estimator 
of $r(x^{(N)},y^{(M)}|\omega(\theta,0))$ by $\hat{\theta}(x^{(N)},y^{(M)})$.
We denote embeddings of $\hat{\theta}(x^{(N)})$ and $\hat{\theta}(x^{(N)},y^{(M)})$
into parameter $\omega$
by $\hat{\omega}_{m}(x^{(N)})$ and $\hat{\omega}_{m}(x^{(N)},y^{(M)})$,
respectively.
We denote 
the best approximating point of $\omega^{*}$ with respect to $r(x^{(N)},y^{(M)}|\omega(\theta,0))$
by $\theta^{(r)}$.
In other words,
$\theta^{(r)}$ is defined by
\begin{eqnarray}
	\theta^{(r)}=\mathop{\mathrm{argmax}}_{\theta\in\Theta_{m}}\mathrm{E}_{\omega^{*}}[\log r(x^{(N)},y^{(M)}|\omega(\theta,0))].
\end{eqnarray}
In the appendix,
we write $\omega^{(p)}$ and $\omega^{(r)}$ instead of $\omega(\theta^{(p)},0)$ and $\omega(\theta^{(r)},0)$,
respectively.
We write $\xi^{(p)}$ instead of $\xi(\theta^{(p)},0)$.

We denote 
the $(a,b)$-components of
the observed Fisher information matrices of
$p(x^{(N)}|\omega(\theta,0))$ and $r(x^{(N)},y^{(M)}|\omega(\theta,0))$
by $\hat{G}^{(p)}_{ab}(\hat{\theta}(x^{(N)}))$ and $\hat{G}^{(r)}_{ab}(\hat{\theta}(x^{(N)},y^{(M)}))$,
respectively.
We denote the stochastic large and small orders with respect to the distribution with the parameter $\omega$
by $\mathrm{O}_{\omega}$ and $\mathrm{o}_{\omega}$,
respectively.

\begin{lem}
\label{Connectionformula}
Under local misspecification,
the following two equations hold:
for $a\in\{1,\ldots,d_{m}\}$
\begin{eqnarray}
h^{a}=-g^{(p)ab}_{m}(\theta^{(p)})g^{(p)}_{b\kappa}(\xi^{(p)})h^{\kappa}+\mathrm{O}(1/\sqrt{N})
\label{h_decomp}
\end{eqnarray}
and
\begin{eqnarray}
	\theta^{(r)a}_{m}-\theta^{(p)a}_{m}=g^{(r)ab}_{m}(\theta^{(p)})g^{(r)}_{bs}(\omega^{(p)})\frac{h^{s}}{\sqrt{N}}+\mathrm{O}(1/N).
\label{bestapproximatingpoints}
\end{eqnarray}
\end{lem}

\begin{proof}
First, we show that the former equation holds.
From (\ref{Localmisspec_omega}),
we obtain for $i\in\{1,\ldots,N\}$,
\begin{eqnarray}
p(x_{i}|\omega^{*})=p(x_{i}|\omega^{(p)})
\left[1+\partial_{s}\log p(x_{i}|\omega^{(p)})
\frac{h^{s}}{\sqrt{N}}+\mathrm{O}_{\omega^{(p)}}(1/N)\right],
\label{p_exp}
\end{eqnarray}
and
for $j\in\{1,\ldots,M\}$,
\begin{eqnarray}
q(y_{j}|\omega^{*})=q(y_{j}|\omega^{(p)})
\left[1+\partial_{s}\log q(y_{j}|\omega^{(p)})
\frac{h^{s}}{\sqrt{N}}+\mathrm{O}_{\omega^{(p)}}(1/N)\right],
\label{q_exp}
\end{eqnarray}
respectively.

Consider the definition of $\omega^{(p)}$:
\begin{eqnarray}
\frac{1}{\sqrt{N}}\mathrm{E}_{\omega^{*}}
\left[\partial_{a}\log p(x^{(N)}|\omega^{(p)})\right]
&=&0.
\label{def_p_bestapp}
\end{eqnarray}
From the independence of $x^{(N)}$ and from (\ref{p_exp}),
the LHS in (\ref{def_p_bestapp}) is expanded as
\begin{eqnarray}
&{}&
\hspace{-7mm}
\frac{1}{\sqrt{N}}\mathrm{E}_{\omega^{*}}
\left[\partial_{a}\log p(x^{(N)}|\omega^{(p)})\right]
\nonumber\\
&=&\frac{1}{\sqrt{N}}\mathop{\Sigma}_{i=1}^{N}
\mathrm{E}_{\omega^{*}}\left[\partial_{a}\log p(x_{i}|\omega^{(p)})\right]
\nonumber\\
&=&\frac{1}{\sqrt{N}}\mathop{\Sigma}_{i=1}^{N}
\mathrm{E}_{\omega^{(p)}}
\left[\left\{1+\partial_{s}\log p(x_{i}|\omega^{(p)})\frac{h^{s}}{\sqrt{N}}
+\mathrm{O}_{\omega^{(p)}}(1/N)\right\}
\{\partial_{a}\log p(x_{i}|\omega^{(p)})\}
\right]
\nonumber\\
&=&\frac{1}{\sqrt{N}}\mathop{\Sigma}_{i=1}^{N}
\mathrm{E}_{\omega^{(p)}}
\left[\partial_{a}\log p(x_{i}|\omega^{(p)})\right]
\nonumber\\
&&
+
\mathop{\Sigma}_{i=1}^{N}
\mathrm{E}_{\omega^{(p)}}
\left[\partial_{s}\log p(x_{i}|\omega^{(p)})
\partial_{a}\log p(x_{i}|\omega^{(p)})
\right]\frac{h^{s}}{N}
+\mathrm{O}(1/\sqrt{N})
\nonumber\\
&=&\frac{1}{N}g^{(p)}_{as}(\omega^{(p)})h^{s}+\mathrm{O}(1/\sqrt{N}).
\label{expansion_p_bestapp}
\end{eqnarray}
By comparing (\ref{def_p_bestapp}) with (\ref{expansion_p_bestapp}) up to constant order,
we obtain
\begin{eqnarray*}
\frac{1}{N}g^{(p)}_{as}(\omega^{(p)})h^{s}=\mathrm{O}(1/\sqrt{N}).
\end{eqnarray*}
By the reparameterization of $\omega$ to $\xi$,
we obtain 
\begin{eqnarray}
\frac{1}{N}g^{(p)}_{a\alpha}(\xi^{(p)})h^{\alpha}=\mathrm{O}(1/\sqrt{N}).
\end{eqnarray}
Thus we obtain (\ref{h_decomp}).

Next, we show the latter equation holds.
Consider the definition of $\omega^{(r)}$:
\begin{eqnarray}
\frac{1}{\sqrt{N}}\mathrm{E}_{\omega^{*}}
\left[\partial_{a}\log r(x^{(N)},y^{(M)}|\omega^{(r)})\right]
&=&0.
\label{def_r_bestapp}
\end{eqnarray}
From the independence of $x^{(N)}$ and $y^{(M)}$,
from (\ref{p_exp}) and (\ref{q_exp}),
and from the Taylor expansions 
of $\partial_{a}\log p(x_{i}|\omega^{(r)})$
and $\partial_{a}\log q(y_{j}|\omega^{(r)})$
around $\omega^{(p)}$,
the LHS in (\ref{def_r_bestapp}) is expanded as
\begin{eqnarray}
&{}&\hspace{-7mm}\frac{1}{\sqrt{N}}\mathrm{E}_{\omega^{*}}
\left[\partial_{a}\log r(x^{(N)},y^{(M)}|\omega^{(r)})\right]
\nonumber\\
&=&\frac{1}{\sqrt{N}}
\mathop{\Sigma}_{i=1}^{N}
\mathrm{E}_{\omega^{*}}\left[
\partial_{a}\log p(x_{i}|\omega^{(r)})\right]
+
\frac{1}{\sqrt{N}}
\mathop{\Sigma}_{j=1}^{M}
\mathrm{E}_{\omega^{*}}\left[
\partial_{a}\log q(y_{j}|\omega^{(r)})\right]
\nonumber\\
&=&
\frac{1}{\sqrt{N}}
\Sigma_{i=1}^{N}
\mathrm{E}_{\omega^{(p)}}\left[
\left\{
1+\partial_{s}\log p(x_{i}|\omega^{(p)})\frac{h^{s}}{\sqrt{N}}
+\mathrm{O}_{\omega^{(p)}}(1/N)
\right\}
\partial_{a}\log p(x_{i}|\omega^{(r)})
\right]
\nonumber\\
&&+
\frac{1}{\sqrt{N}}
\Sigma_{j=1}^{M}
\mathrm{E}_{\omega^{(p)}}\left[
\left\{
1+\partial_{s}\log q(y_{j}|\omega^{(p)})\frac{h^{s}}{\sqrt{N}}
+\mathrm{O}_{\omega^{(p)}}(1/N)
\right\}
\partial_{a}\log q(y_{j}|\omega^{(r)})
\right]
\nonumber\\
&=&
\frac{1}{\sqrt{N}}
\mathop{\Sigma}_{i=1}^{N}
\mathrm{E}_{\omega^{(p)}}
\left[
\left\{
1+\partial_{s}\log p(x_{i}|\omega^{(p)})\frac{h^{s}}{\sqrt{N}}
\right\}
\right.\nonumber\\
&&\left.\quad\quad\quad
\times\bigl\{
\partial_{a}\log p(x_{i}|\omega^{(p)})
+\partial_{ab}\log p(x_{i}|\omega^{(p)})
(\theta^{(r)b}-\theta^{(p)b})
+\mathrm{O}_{\omega^{(p)}}(\vert\vert \theta^{(r)}-\theta^{(p)}\vert\vert^{2})
\bigr\}
\right]
\nonumber\\
&&+
\frac{1}{\sqrt{N}}
\mathop{\Sigma}_{j=1}^{M}
\mathrm{E}_{\omega^{(p)}}
\left[
\left\{
1+\partial_{s}\log q(y_{j}|\omega^{(p)})\frac{h^{s}}{\sqrt{N}}
\right\}
\right.\nonumber\\
&&\left.\quad\quad\quad
\times\bigl\{
\partial_{a}\log q(y_{j}|\omega^{(p)})
+\partial_{ab}\log q(y_{j}|\omega^{(p)})
(\theta^{(r)b}-\theta^{(p)b})
+\mathrm{O}_{\omega^{(p)}}(\vert\vert \theta^{(r)}-\theta^{(p)}\vert\vert^{2})
\bigr\}
\right]
\nonumber\\
&=&
\frac{1}{\sqrt{N}}
\mathrm{E}_{\omega^{(p)}}
\left[
\partial_{a}\log r(x^{(N)},y^{(M)}|\omega^{(p)})
\right]
\nonumber\\
&&+g^{(r)}_{sa}(\omega^{(p)})\frac{h^{s}}{N}
-\frac{1}{\sqrt{N}}g^{(r)}_{ab}(\theta^{(p)})(\theta^{(r)b}-\theta^{(p)b})
+\mathrm{O}(\sqrt{N}\vert\vert\theta^{(r)}-\theta^{(p)}\vert\vert^{2}).
\label{expansion_r_bestapp}
\end{eqnarray}
Thus,
we obtain (\ref{bestapproximatingpoints})
by comparing (\ref{def_r_bestapp}) with (\ref{expansion_r_bestapp}) up to constant order.
\end{proof}

\begin{lem}
\label{MLEs}
Under local misspecification,
the following equations hold:
for $a\in\{1,\ldots,d_{m}\}$,
\begin{eqnarray}
\hat{\theta}^{a}(x^{(N)},y^{(M)})-\theta^{(r)a}
&=&g^{(r)ab}_{m}(\theta^{(r)})
\partial_{b}\log r(x^{(N)},y^{(M)}|\omega^{(r)})
+\mathrm{O}_{\omega^{*}}(1/N)
\label{locallinear_sub_r}
\end{eqnarray}
and
\begin{eqnarray}
\hat{\theta}^{a}(x^{(N)})-\theta^{(p)a}
&=&g^{(p)ab}_{m}(\theta^{(p)})
\partial_{b}\log p(x^{(N)}|\omega^{(p)})
+\mathrm{O}_{\omega^{*}}(1/N),
\label{locallinear_sub_p}
\end{eqnarray}
respectively.

For $s\in\{1,\ldots,d_{\mathrm{full}}\}$,
\begin{eqnarray}
\hat{\omega}^{s}(x^{(N)},y^{(M)})-\omega^{*s}&=&g^{(r)st}(\omega^{*})
\partial_{t}\log r(x^{(N)},y^{(M)}|\omega^{*})+\mathrm{O}_{\omega^{*}}(1/N)
\label{locallinear_full_r}
\end{eqnarray}
and
\begin{eqnarray}
\hat{\omega}^{s}(x^{(N)})-\omega^{*s}&=&g^{(p)st}(\omega^{*})
\partial_{t}\log p(x^{(N)}|\omega^{*})+\mathrm{O}_{\omega^{*}}(1/N),
\label{locallinear_full_p}
\end{eqnarray}
\end{lem}
respectively.

\begin{proof}
Consider the estimative equations:
\begin{eqnarray}
\partial_{a}\log r(x^{(N)},y^{(M)}|\hat{\omega}_{m}(x^{(N)},y^{(M)}))&=&0
\label{estimative_r_restricted}
\end{eqnarray}
and
\begin{eqnarray}
\partial_{a}\log p(x^{(N)}|\hat{\omega}_{m}(x^{(N)}))&=&0.
\label{estimative_p_restricted}
\end{eqnarray}
We apply the Taylor expansions around 
$\omega^{(r)}$ and $\omega^{(p)}$
to equations (\ref{estimative_r_restricted}) and (\ref{estimative_p_restricted}),
respectively.
Since $\partial_{ab}\log r(x^{(N)},y^{(M)}|\omega^{(r)})+g^{(r)}_{ab}(\theta^{(r)})=\mathrm{O}_{\omega^{(r)}}(\sqrt{N})$
and $\omega^{*}-\omega^{(r)}=\mathrm{O}(1/\sqrt{N})$,
we obtain the following expansion:
\begin{eqnarray*}
&{}&\hspace{-7mm}\partial_{a}\log r(x^{(N)},y^{(M)}|\hat{\omega}_{m}(x^{(N)},y^{(M)}))
\nonumber\\
&=&
\partial_{a}\log r(x^{(N)},y^{(M)}|\omega^{(r)})
+
\partial_{ab}\log r(x^{(N)},y^{(M)}|\omega^{(r)})
\left\{\hat{\theta}^{b}(x^{(N)},y^{(M)})-\theta^{(r)b}\right\}
\nonumber\\
&&+
\mathrm{O}_{\omega^{(r)}}(\sqrt{N}
\vert\vert 
\hat{\theta}(x^{(N)},y^{(M)})-\theta^{(r)}
\vert\vert)
+
\mathrm{O}_{\omega^{(r)}}(
N\vert\vert
\hat{\theta}(x^{(N)},y^{(M)})-\theta^{(r)}
\vert\vert^{2})
\nonumber\\
&=&
\partial_{a}\log r(x^{(N)},y^{(M)}|\omega^{(r)})
-
g^{(r)}_{ab}(\theta^{(r)})\left\{\hat{\theta}^{b}(x^{(N)},y^{(M)})-\theta^{(r)b}\right\}
+\mathrm{O}_{\omega^{*}}(1).
\end{eqnarray*}
Likewise,
we obtain the following expansion:
\begin{eqnarray*}
&{}&\hspace{-7mm}\partial_{a}\log p(x^{(N)}|\hat{\omega}_{m}(x^{(N)}))
\nonumber\\
&=&
\partial_{a}\log p(x^{(N)}|\omega^{(p)})
+
\partial_{ab}\log p(x^{(N)}|\omega^{(p)})
\left\{\hat{\theta}^{b}(x^{(N)})-\theta^{(p)b}\right\}
\nonumber\\
&&+
\mathrm{O}_{\omega^{(p)}}(\sqrt{N}
\vert\vert 
\hat{\theta}(x^{(N)})-\theta^{(p)}
\vert\vert)
+
\mathrm{O}_{\omega^{(p)}}(N\vert\vert
\hat{\theta}(x^{(N)})-\theta^{(p)}
\vert\vert^{2})
\nonumber\\
&=&
\partial_{a}\log p(x^{(N)}|\omega^{(p)})
-g^{(p)}_{ab}(\theta^{(p)})
\left\{\hat{\theta}^{b}(x^{(N)})-\theta^{(p)b}\right\}
+
\mathrm{O}_{\omega^{*}}(1).
\end{eqnarray*}

Thus, we obtain (\ref{locallinear_sub_r}) and (\ref{locallinear_sub_p}).
Equations (\ref{locallinear_full_r}) and (\ref{locallinear_full_p})
immediately follow from the estimative equations of $\hat{\omega}$.
For example,
see Theorem 5.39 in \citet{vanderVaart}.
\end{proof}

\begin{proof}[Proof of Theorem \ref{Risk_decomp}]
We prove Theorem \ref{Risk_decomp} by using the above lemmas.
Consider the following decomposition of the Kullback--Leibler risk:
\begin{eqnarray}
R(\omega^{*},q_{m,\pi})&=&
\mathrm{E}_{\omega^{*}}
\left[\log\frac{r(x^{(N)},y^{(M)}|\omega^{*})}{r_{m,\pi}(x^{(N)},y^{(M)})}\right]
-\mathrm{E}_{\omega^{*}}
\left[\log\frac{p(x^{(N)}|\omega^{*})}{p_{m,\pi}(x^{(N)})}\right].
\label{KL_decomp}
\end{eqnarray}
The marginal distributions $r_{m,\pi}(x^{(N)},y^{(M)})$ and $p_{m,\pi}(x^{(N)})$
are expanded as
\begin{eqnarray}
r_{m,\pi}(x^{(N)},y^{(M)})
&=&
(2\pi)^{d_{m}/2}
\frac{\pi(\hat{\theta}(x^{(N)},y^{(M)}))}{\vert \hat{G}^{(r)}(\hat{\theta}(x^{(N)},y^{(M)}))\vert^{1/2}}
\nonumber\\
&&\quad\quad
\times r(x^{(N)},y^{(M)}|\hat{\omega}_{m}(x^{(N)},y^{(M)}))
\{1+\mathrm{o}(1)\}
\label{marginal_r}
\end{eqnarray}
and
\begin{eqnarray}
p_{m,\pi}(x^{(N)})
&=&
(2\pi)^{d_{m}/2}
\frac{\pi(\hat{\theta}(x^{(N)}))}{\vert \hat{G}^{(p)}(\hat{\theta}(x^{(N)}))\vert^{1/2}}
p(x^{(N)}|\hat{\omega}_{m}(x^{(N)}))
\{1+\mathrm{o}(1)\},
\label{marginal_p}
\end{eqnarray}
respectively.
See p. 117 in \citet{Ghosh_Introduction_Bayes}.

By using the marginal expansions (\ref{marginal_r}) and (\ref{marginal_p}),
the above decomposition is expanded as
\begin{eqnarray}
&{}&\hspace{-7mm}R(\omega^{*},q_{m,\pi})
\nonumber\\
&=&\mathrm{E}_{\omega^{*}}
\left[
\log\frac{r(x^{(N)},y^{(M)}|\omega^{*})}{r(x^{(N)},y^{(M)}|\hat{\omega}_{m}(x^{(N)},y^{(M)}))}
\right]
-\mathrm{E}_{\omega^{*}}
\left[\log\frac{p(x^{(N)}|\omega^{*})}{p(x^{(N)}|\hat{\omega}_{m}(x^{(N)}))}
\right]
\nonumber\\
&&+\mathrm{E}_{\omega^{*}}\left[
\frac{1}{2}\log
\frac{\vert \hat{G}^{(r)}(\hat{\theta}(x^{(N)},y^{(M)})) \vert}
{\vert \hat{G}^{(p)}(\hat{\theta}(x^{(N)}))\vert}
\right]
-\mathrm{E}_{\omega^{*}}\left[
\log\frac{\pi(\hat{\theta}(x^{(N)},y^{(M)}))}{\pi(\hat{\theta}(x^{(N)}))}
\right]
+\mathrm{o}(1).
\label{KL_marginal_decomp}
\end{eqnarray}
From (\ref{Localmisspec_omega}), (\ref{bestapproximatingpoints}), and (\ref{locallinear_sub_r}),
the following equation holds:
\begin{eqnarray}
&{}&\hspace{-7mm}\hat{\omega}^{s}_{m}(x^{(N)},y^{(M)})-\omega^{*s}
\nonumber\\
&=&\omega^{s}_{m}(x^{(N)},y^{(M)})-\omega^{(r)s}
+\omega^{(r)s}-\omega^{(p)s}
+\omega^{(p)s}-\omega^{*s}
\nonumber\\
&=&
\frac{\partial\omega^{s}}{\partial\theta^{a}}(\theta^{(r)})
g^{(r)ab}_{m}(\theta^{(r)})
\partial_{b}\log r(x^{(N)},y^{(M)}|\omega^{(r)})
\nonumber\\
&&+\frac{\partial\omega^{s}}{\partial\theta^{a}}(\theta^{(p)})
g^{(r)ab}_{m}(\theta^{(p)})g^{(r)}_{bs}(\omega^{(p)})\frac{h^{s}}{\sqrt{N}}
-\frac{h^{s}}{\sqrt{N}}
+\mathrm{O}_{\omega^{*}}(1/N).
\label{MLE_linear_imb_r}
\end{eqnarray}

First,
consider the first term in (\ref{KL_marginal_decomp}).
By using the Taylor expansion,
we expand the negative of the first term as 
\begin{eqnarray}
&{}&\hspace{-7mm}\mathrm{E}_{\omega^{*}}\left[\log
\frac{r(x^{(N)},y^{(M)}|\hat{\omega}_{m}(x^{(N)},y^{(M)}))}{r(x^{(N)},y^{(M)}|\omega^{*})}
\right]
\nonumber\\
&=&\mathrm{E}_{\omega^{*}}
\left[
\partial_{s}\log r(x^{(N)},y^{(M)}|\omega^{*})
\{\hat{\omega}^{s}_{m}(x^{(N)},y^{(M)})-\omega^{*s}\}
\right]
\nonumber\\
&&+\frac{1}{2}\mathrm{E}_{\omega^{*}}
\left[
\partial_{st}\log r(x^{(N)},y^{(M)}|\omega^{*})
\{\hat{\omega}^{s}_{m}(x^{(N)},y^{(M)})-\omega^{*s}\}
\{\hat{\omega}^{t}_{m}(x^{(N)},y^{(M)})-\omega^{*t}\}
\right]
\nonumber\\
&&+\mathrm{o}(1).
\label{Firstterm_r_decomp}
\end{eqnarray}
From 
the Taylor expansion 
of $\partial_{b}\log r(x^{(N)},y^{(M)}|\omega^{(r)})$
around $\omega^{*}$,
we obtain the following equation
for the first term in (\ref{Firstterm_r_decomp}):
\begin{eqnarray}
&{}&\hspace{-7mm}
\mathrm{E}_{\omega^{*}}[\partial_{s}\log r(x^{(N)},y^{(M)}|\omega^{*})\{\hat{\omega}^{s}_{m}(x^{(N)},y^{(M)})-\omega^{*s}\}]
\nonumber\\
&=&\mathrm{E}_{\omega^{*}}\left[
\partial_{s}\log r(x^{(N)},y^{(M)}|\omega^{*})
\left\{
\frac{\partial\omega^{s}}{\partial\theta^{a}}(\theta^{(r)})
g^{(r)ab}_{m}(\theta^{(r)})
\partial_{b}\log r(x^{(N)},y^{(M)}|\omega^{(r)})
\right.
\right.
\nonumber\\
&&\left.\left.\quad\quad\quad+\frac{\partial\omega^{s}}{\partial\theta^{a}}(\theta^{(p)})
g^{(r)ab}_{m}(\theta^{(p)})g^{(r)}_{bs}(\omega^{(p)})\frac{h^{s}}{\sqrt{N}}
-\frac{h^{s}}{\sqrt{N}}
\right\}
\right]
+\mathrm{o}(1)
\nonumber\\
&=&d_{m}+\mathrm{o}(1).
\label{r_exp_1}
\end{eqnarray}

From (\ref{MLE_linear_imb_r}),
we expand the second term in (\ref{Firstterm_r_decomp}) as
\begin{eqnarray}
&{}&\hspace{-7mm}\frac{1}{2}\mathrm{E}_{\omega^{*}}
[
\partial_{st}\log r(x^{(N)},y^{(M)}|\omega^{*})
\{\hat{\omega}^{s}_{m}(x^{(N)},y^{(M)})-\omega^{*s}\}
\{\hat{\omega}^{t}_{m}(x^{(N)},y^{(M)})-\omega^{*t}\}
]
\nonumber\\
&=&-\frac{1}{2}g^{(r)}_{st}(\omega^{*})
\mathrm{E}_{\omega^{*}}
[
\{\hat{\omega}^{s}_{m}(x^{(N)},y^{(M)})-\omega^{*s}\}
\{\hat{\omega}^{t}_{m}(x^{(N)},y^{(M)})-\omega^{*t}\}
]
+\mathrm{o}(1)
\nonumber\\
&=&-\frac{1}{2}g^{(r)}_{st}(\omega^{*})
\mathrm{E}_{\omega^{*}}
\left[
\left\{
\frac{\partial\omega^{s}}{\partial\theta^{a}}(\theta^{(r)})
g^{(r)ab}_{m}(\theta^{(r)})
\partial_{b}\log r(x^{(N)},y^{(M)}|\omega^{(r)})
\right.\right.
\nonumber\\
&&\left.\left.\quad\quad\quad\quad\quad\quad\quad
+\frac{\partial\omega^{s}}{\partial\theta^{a}}(\theta^{(p)})
g^{(r)ab}_{m}(\theta^{(p)})g^{(r)}_{bs}(\omega^{(p)})\frac{h^{s}}{\sqrt{N}}
-\frac{h^{s}}{\sqrt{N}}
\right\}
\right.
\nonumber\\
&&\left.\quad\quad\quad\quad\quad
\quad\times\left\{
\frac{\partial\omega^{t}}{\partial\theta^{c}}(\theta^{(r)})
g^{(r)cd}_{m}(\theta^{(r)})
\partial_{d}\log r(x^{(N)},y^{(M)}|\omega^{(r)})
\right.\right.
\nonumber\\
&&\left.\left.\quad\quad\quad\quad\quad\quad\quad
+\frac{\partial\omega^{t}}{\partial\theta^{c}}(\theta^{(p)})
g^{(r)cd}_{m}(\theta^{(p)})g^{(r)}_{dt}(\omega^{(p)})\frac{h^{t}}{\sqrt{N}}
-\frac{h^{t}}{\sqrt{N}}
\right\}
\right]
+\mathrm{o}(1)
\nonumber\\
&=&
-\frac{1}{2}g^{(r)}_{ac}(\omega^{(r)})
g^{(r)ab}_{m}(\theta^{(r)})
g^{(r)cd}_{m}(\theta^{(r)})
\mathrm{E}_{\omega^{*}}
\left[
\partial_{b}\log r(x^{(N)},y^{(M)}|\omega^{(r)})
\partial_{d}\log r(x^{(N)},y^{(M)}|\omega^{(r)})
\right]
\nonumber\\
&&-\frac{1}{2}g^{(r)}_{st}(\omega^{*})\frac{h^{s}h^{t}}{N}
\nonumber\\
&&-\frac{1}{2}g^{(r)}_{st}(\omega^{*})
\frac{\partial\omega^{s}}{\partial\theta^{a}}(\theta^{(p)})
\frac{\partial\omega^{t}}{\partial\theta^{c}}(\theta^{(p)})
g^{(r)ab}_{m}(\theta^{(p)})g^{(r)}_{bu}(\omega^{(p)})
\frac{h^{u}}{\sqrt{N}}
g^{(r)cd}_{m}(\theta^{(p)})g^{(r)}_{dv}(\omega^{(p)})
\frac{h^{v}}{\sqrt{N}}
\nonumber\\
&&+g^{(r)}_{st}(\omega^{*})\frac{h^{s}}{\sqrt{N}}
\frac{\partial\omega^{t}}{\partial\theta^{c}}(\theta^{(p)})
g^{(r)cd}_{m}(\theta^{(p)})
g^{(r)}_{dv}(\omega^{(p)})
\frac{h^{v}}{\sqrt{N}}
+\mathrm{o}(1).
\label{Quadraform_r}
\end{eqnarray}
From the independence of $x^{(N)}$ and $y^{(M)}$ and from the Taylor expansions of $\partial_{a}\log p(x_{i}|\omega^{(r)})$
and $\partial_{a}\log q(y_{j}|\omega^{(r)})$
around $\omega^{*}$,
\begin{eqnarray}
&{}&\hspace{-7mm}\mathrm{E}_{\omega^{*}}\left[
\partial_{a}\log r(x^{(N)},y^{(M)}|\omega^{(r)})
\partial_{b}\log r(x^{(N)},y^{(M)}|\omega^{(r)})
\right]
\nonumber\\
&=&
\Sigma_{i=1}^{N}\mathrm{E}_{\omega^{*}}[
	\partial_{a}\log p(x_{i}|\omega^{(r)})\partial_{b}\log p(x_{i}|\omega^{(r)})
]
+
\Sigma_{j=1}^{M}\mathrm{E}_{\omega^{*}}[
	\partial_{a}\log q(y_{j}|\omega^{(r)})\partial_{b}\log q(y_{j}|\omega^{(r)})
]
\nonumber\\
&&+
\Sigma_{i\neq k}^{N}\mathrm{E}_{\omega^{*}}[
	\partial_{a}\log p(x_{i}|\omega^{(r)})\partial_{b}\log p(x_{k}|\omega^{(r)})
]
+
\Sigma_{j\neq l}^{M}\mathrm{E}_{\omega^{*}}[
	\partial_{a}\log q(y_{j}|\omega^{(r)})\partial_{b}\log q(y_{l}|\omega^{(r)})
]
\nonumber\\
&&+
2\Sigma_{i,j}^{i=N,j=M}\mathrm{E}_{\omega^{*}}[
	\partial_{a}\log p(x_{i}|\omega^{(r)})\partial_{b}\log q(y_{j}|\omega^{(r)})
]
\nonumber\\
&=&g^{(r)}_{ab}(\omega^{(r)})+\mathrm{O}(\sqrt{N}).
\label{product_score_r}
\end{eqnarray}
By substituting (\ref{product_score_r})
into the first term in (\ref{Quadraform_r}),
we obtain the following further expansion of (\ref{Quadraform_r}):
\begin{eqnarray}
&&\frac{1}{2}\mathrm{E}_{\omega^{*}}[\partial_{st}\log r(x^{(N)},y^{(M)}|\omega^{*})
\{\omega^{*s}-\hat{\omega}^{s}_{m}(x^{(N)},y^{(M)})\}
\{\omega^{*t}-\hat{\omega}^{t}_{m}(x^{(N)},y^{(M)})\}]
\nonumber\\
&=&-\frac{1}{2}g^{(r)}_{st}(\omega^{(p)})\frac{h^{s}h^{t}}{N}+\frac{1}{2}g^{(r)ab}_{m}(\theta^{(p)})g^{(r)}_{as}(\omega^{(p)})g^{(r)}_{bt}(\omega^{(p)})\frac{h^{s}h^{t}}{N}
-\frac{1}{2}d_{m}+\mathrm{o}(1).
\label{r_exp_2}
\end{eqnarray}
By combining (\ref{r_exp_1}) and (\ref{r_exp_2}),
we obtain the following equation for (\ref{Firstterm_r_decomp}):
\begin{eqnarray}
&{}&\hspace{-7mm}\mathrm{E}_{\omega^{*}}
\left[
\log\frac{r(x^{(N)},y^{(M)}|\hat{\omega}_{m}(x^{(N)},y^{(M)}))}{r(x^{(N)},y^{(M)}|\omega^{*})}
\right]
\nonumber\\
&=&-\frac{1}{2}\bigl[g^{(r)}_{st}(\omega^{(p)})
-g^{(r)ab}_{m}(\theta^{(p)})g^{(r)}_{as}(\omega^{(p)})g^{(r)}_{bt}(\omega^{(p)})\bigr]
\frac{h^{s}h^{t}}{N}
+\frac{1}{2}d_{m}
+\mathrm{o}(1).
\label{Firstterm_r_decomp_final}
\end{eqnarray}

Next,
consider the second term in (\ref{KL_marginal_decomp}).
The estimator $\hat{\omega}_{m}(x^{(N)})$ is expanded as
\begin{eqnarray}
\hat{\omega}^{s}_{m}(x^{(N)})-\omega^{*s}
&=&
\frac{\partial\omega^{s}}{\partial\theta^{a}}(\theta^{(p)})
g^{(p)ab}_{m}(\theta^{(p)})\partial_{b}\log p(x^{(N)}|\omega^{(p)})
-\frac{h^{s}}{\sqrt{N}}
+\mathrm{O}_{\omega^{*}}(1/N).
\label{MLE_linear_imb_p}
\end{eqnarray}
By using the Taylor expansion,
we expand the negative of the second term in (\ref{KL_marginal_decomp}) as
\begin{eqnarray}
	&&\mathrm{E}_{\omega^{*}}\left[\log\frac{p(x^{N)}|\hat{\omega}_{m}(x^{(N)}))}{p(x^{(N)}|\omega^{*})}\right]
\nonumber\\
&=&\mathrm{E}_{\omega^{*}}[\partial_{s}\log p(x^{(N)}|\omega^{*})\{\hat{\omega}^{s}_{m}(x^{(N)})-\omega^{*s}\}]
\nonumber\\
&&+\frac{1}{2}\mathrm{E}_{\omega^{*}}[\partial_{st}\log p(x^{(N)}|\omega^{*})
\{\hat{\omega}^{s}_{m}(x^{(N)})-\omega^{*s}\}
\{\hat{\omega}^{t}_{m}(x^{(N)})-\omega^{*t}\}]
+\mathrm{o}(1).
\label{Secondterm_p_decomp}
\end{eqnarray}

From (\ref{MLE_linear_imb_p}),
we obtain
\begin{eqnarray}
&{}&\hspace{-7mm}
\mathrm{E}_{\omega^{*}}
\left[
\partial_{s}\log p(x^{(N)}|\omega^{*})
\{\hat{\omega}^{s}_{m}(x^{(N)})-\omega^{*s}\}
\right]
\nonumber\\
&=&
\mathrm{E}_{\omega^{*}}
\left[\partial_{s}\log p(x^{(N)}|\omega^{*})
\left\{\frac{\partial\omega^{s}}{\partial\theta^{a}}(\theta^{(p)})
g^{(p)ab}_{m}(\theta^{(p)})\partial_{b}\log p(x^{(N)}|\omega^{(p)})
-\frac{h^{s}}{\sqrt{N}}+\mathrm{O}_{\omega^{*}}(1/N)
\right\}\right]
\nonumber\\
&=&g^{(p)}_{ab}(\omega^{(p)})g^{(p)ab}_{m}(\theta^{(p)})+\mathrm{o}(1)
\nonumber\\
&=&d_{m}+\mathrm{o}(1)
\label{p_exp_1}
\end{eqnarray}
and
\begin{eqnarray}
&{}&\hspace{-7mm}
\mathrm{E}_{\omega^{*}}[g^{(p)}_{st}(\omega^{*})
\{\hat{\omega}_{m}(x^{(N)})-\omega^{*s}\}\{\hat{\omega}_{m}(x^{(N)})-\omega^{*t}\}]
\nonumber\\
&=&g^{(p)}_{st}(\omega^{*})\frac{h^{s}h^{t}}{N}+g^{(p)}_{st}(\omega^{*})\frac{\partial\omega^{s}}{\partial\theta^{a}}(\theta^{(p)})
\frac{\partial\omega^{t}}{\partial\theta^{b}}(\theta^{(p)})g^{(p)ac}_{m}(\theta^{(p)})g^{(p)bd}_{m}(\theta^{(p)})g^{(p)}_{bd}(\omega^{(p)})
+\mathrm{o}(1)
\nonumber\\
&=&g^{(p)}_{st}(\omega^{*})\frac{h^{s}h^{t}}{N}+d_{m}+\mathrm{o}(1).
\label{p_exp_2}
\end{eqnarray}
From (\ref{p_exp_1}) and (\ref{p_exp_2}),
we obtain the following equation for (\ref{Secondterm_p_decomp}):
\begin{eqnarray}
\mathrm{E}_{\omega^{*}}
\left[
	\log\frac{p(x^{(N)}|\hat{\omega}_{m}(x^{(N)}))}{p(x^{(N)}|\omega^{*})}
\right]
&=&-\frac{1}{2}g^{(p)}_{st}(\omega^{*})\frac{h^{s}h^{t}}{N}+\frac{1}{2}d_{m}+\mathrm{o}(1).
\label{Secondterm_p_decomp_final}
\end{eqnarray}
The Taylor expansions around $\theta^{(p)}$ and equation (\ref{bestapproximatingpoints})
show that the third and fourth terms
in (\ref{KL_marginal_decomp}) are equal to $\mathrm{o}(1)$.
Thus,
from (\ref{Firstterm_r_decomp_final}) and (\ref{Secondterm_p_decomp_final}),
the Kullback--Leibler risk $R(\omega^{*},q_{m,\pi})$ is expanded as
\begin{eqnarray}
&{}&\hspace{-7mm}R(\omega^{*},q_{m,\pi})
\nonumber\\
&=&\frac{1}{2N}
\left[
g^{(r)}_{st}(\omega^{*})
-g^{(p)}_{st}(\omega^{*})-g^{(r)ab}_{m}(\theta^{(p)})g^{(r)}_{sa}(\omega^{(p)})g^{(r)}_{tb}(\omega^{(p)})
\right]
h^{s}h^{t}
\nonumber\\
&&+\frac{1}{2}\log\frac{|g^{(r)}(\theta^{(p)})|}{|g^{(p)}(\theta^{(p)})|}
+\mathrm{o}(1).
\label{L_exp_final}
\end{eqnarray}
Note that this is invariant up to $\mathrm{o}(1)$ under the reparameterization of $\omega$.

Let $P$ be a matrix whose $(\alpha,\beta)$-component is given by
\begin{eqnarray}
P_{\alpha\beta}
&=&g^{(r)}_{\alpha\beta}(\xi^{*})
-g^{(p)}_{\alpha\beta}(\xi^{*})-g^{(r)ab}_{m}(\theta^{(p)})g^{(r)}_{a\alpha}(\xi^{(p)})g^{(r)}_{b\beta}(\xi^{(p)}).
\end{eqnarray}

To complete the proof of Theorem \ref{Risk_decomp},
we show
\begin{eqnarray}
P_{\alpha\beta}h^{\alpha}h^{\beta}/N &=& S_{\alpha\beta}h^{\alpha}h^{\beta}/N+\mathrm{o}(1).
\label{PS_equivalence}
\end{eqnarray}

From (\ref{h_decomp}),
we obtain
\begin{eqnarray}
&{}&\hspace{-7mm}P_{ab}h^{a}h^{b}
\nonumber\\
&=&
\left\{
g^{(r)}_{ab}(\xi^{(p)})-g^{(p)}_{ab}(\xi^{(p)})
-g^{(r)cd}_{m}(\theta^{(p)})g^{(r)}_{ac}(\xi^{(p)})g^{(r)}_{bd}(\xi^{(p)})
\right\}
h^{a}h^{b}
\nonumber\\
&=&-g^{(p)}_{ab}(\xi^{(p)})h^{a}h^{b}
\nonumber\\
&=&-g^{(p)}_{ab}(\xi^{(p)})
\left\{
	-g^{(p)ac}_{m}(\theta^{(p)})g^{(p)}_{c\kappa}(\xi^{(p)})h^{\kappa}
	+\mathrm{o}(1)
\right\}
\left\{
	-g^{(p)bd}_{m}(\theta^{(p)})g^{(p)}_{d\lambda}(\xi^{(p)})h^{\lambda}
	+\mathrm{o}(1)
\right\}
\nonumber\\
&=&-g^{(p)ab}_{m}(\theta^{(p)})g^{(p)}_{a\kappa}(\xi^{(p)})g^{(p)}_{b\lambda}(\xi^{(p)})h^{\kappa}h^{\lambda}
+\mathrm{o}(N)
\label{P_ab}
\end{eqnarray}
and
\begin{eqnarray}
	&{}&\hspace{-7mm}P_{a\kappa}h^{a}h^{\kappa}
	\nonumber\\
	&=&
	\left\{
g^{(r)}_{a\kappa}(\xi^{(p)})-g^{(p)}_{a\kappa}(\xi^{(p)})
-g^{(r)cd}_{m}(\theta^{(p)})g^{(r)}_{ac}(\xi^{(p)})g^{(r)}_{d\kappa}(\xi^{(p)})
	\right\}h^{a}h^{\kappa}
	\nonumber\\
	&=&
	\left\{
		g^{(r)}_{a\kappa}(\xi^{(p)})-g^{(p)}_{a\kappa}(\xi^{(p)})-g^{(r)}_{a\kappa}(\xi^{(p)})
	\right\}
	\left\{
		-g^{(p)ae}_{m}(\theta^{(p)})g^{(p)}_{e\lambda}(\xi^{(p)})h^{\lambda}+\mathrm{o}(1)
	\right\}h^{\kappa}
	\nonumber\\
	&=&g^{(p)}_{a\kappa}(\xi^{(p)})g^{(p)ab}_{m}(\theta^{(p)})g^{(p)}_{b\lambda}(\xi^{(p)})h^{\kappa}h^{\lambda}
	+\mathrm{o}(N).
	\label{P_akappa}
\end{eqnarray}
We have
\begin{eqnarray}
	P_{\kappa\lambda}h^{\kappa}h^{\lambda}
	&=&
	\left\{
		g^{(r)}_{\kappa\lambda}(\xi^{(p)})-g^{(p)}_{\kappa\lambda}(\xi^{(p)})
		-g^{(r)ab}_{m}(\theta^{(p)})g^{(r)}_{a\kappa}(\xi^{(p)})g^{(r)}_{b\lambda}(\xi^{(p)})
	\right\}h^{\kappa}h^{\lambda}.
	\label{P_kappalambda}
\end{eqnarray}
From (\ref{P_ab}), (\ref{P_akappa}), and (\ref{P_kappalambda}),
we obtain
\begin{eqnarray}
P_{\alpha\beta}h^{\alpha}h^{\beta}
&=&
P_{ab}h^{a}h^{b}+2P_{a\kappa}h^{a}h^{\kappa}+P_{\kappa\lambda}h^{\kappa}h^{\lambda}
\nonumber\\
&=&
\{
g^{(q)}_{\kappa\lambda}(\xi^{(p)})
+g^{(p)ab}_{m}(\theta^{(p)})g^{(p)}_{a\kappa}(\xi^{(p)})g^{(p)}_{b\lambda}(\xi^{(p)})
\nonumber\\
&&\quad\quad\quad-g^{(r)ab}_{m}(\theta^{(p)})g^{(r)}_{a\kappa}(\xi^{(p)})g^{(r)}_{b\lambda}(\xi^{(p)})\}
h^{\kappa}h^{\lambda}+\mathrm{o}(N).
\label{P_decomp}
\end{eqnarray}

By applying Sherman--Morisson--Woodbury identity to matrix $S$,
the following equation holds:
\begin{eqnarray}
S&=&\left[g^{(q)-1}(\xi^{*})+\begin{pmatrix}g^{(p)-1}_{m}(\theta^{(p)}) & \bigzerou^{\top}_{(d_{\mathrm{full}}-d_{m})\times d_{m}}  \\ \bigzerol_{(d_{\mathrm{full}}-d_{m})\times d_{m}}  & \bigzerol_{(d_{\mathrm{full}}-d_{m})\times (d_{\mathrm{full}}-d_{m})} \end{pmatrix}\right]^{-1}
\nonumber\\
&=&\left[g^{(q)-1}(\xi^{*})+\begin{pmatrix}I \\ \bigzerol_{(d_{\mathrm{full}}-d_{m})\times d_{m}}  \end{pmatrix}g^{(p)-1}_{m}(\theta^{(p)})
\begin{pmatrix}I & \bigzerou^{\top}_{(d_{\mathrm{full}}-d_{m})\times d_{m}} \end{pmatrix}\right]^{-1}
\nonumber\\
&=&g^{(q)}(\xi^{*})
\nonumber\\
&&-g^{(q)}(\xi^{*})
\begin{pmatrix}I \\ \bigzerol_{(d_{\mathrm{full}}-d_{m})\times d_{m}} 
\end{pmatrix}
\left[g^{(p)}_{m}(\theta^{(p)})+\begin{pmatrix}I & \bigzerou^{\top}_{(d_{\mathrm{full}}-d_{m})\times d_{m}}  \end{pmatrix}
g^{(q)}(\xi^{*})\begin{pmatrix}I \\ \bigzerol_{(d_{\mathrm{full}}-d_{m})\times d_{m}} \end{pmatrix}\right]^{-1}
\nonumber\\
&&\quad\quad\quad\quad
\begin{pmatrix}I & \bigzerou^{\top}_{(d_{\mathrm{full}}-d_{m})\times d_{m}}  \end{pmatrix}g^{(q)}(\xi^{*})
\nonumber\\
&=&g^{(q)}(\xi^{*})-g^{(q)}(\xi^{*})\begin{pmatrix}g^{(r)-1}_{m}(\theta^{(p)}) & \bigzerou^{\top}_{(d_{\mathrm{full}}-d_{m})\times d_{m}}  \\ \bigzerol_{(d_{\mathrm{full}}-d_{m})\times d_{m}}  & \bigzerol_{(d_{\mathrm{full}}-d_{m})\times (d_{\mathrm{full}}-d_{m})} \end{pmatrix}g^{(q)}(\xi^{*}),
\end{eqnarray}
where $I$ is the $d_{m}$-dimensional identity matrix,
$\bigzerol_{(d_{\mathrm{full}}-d_{m})\times d_{m}} $ is the $(d_{\mathrm{full}}-d_{m})\times d_{m}$-dimensional zero matrix,
and
$\bigzerol_{(d_{\mathrm{full}}-d_{m})\times (d_{\mathrm{full}}-d_{m})} $ is the $(d_{\mathrm{full}}-d_{m})\times (d_{\mathrm{full}}-d_{m})$-dimensional zero matrix.
From (\ref{h_decomp}),
we obtain 
\begin{eqnarray}
&{}&\hspace{-7mm}S_{ab}h^{a}h^{b}
\nonumber\\
&=&
g^{(q)}_{ac}(\xi^{(p)})g^{(r)cd}_{m}(\theta^{(p)})g^{(r)}_{db}(\xi^{(p)})h^{a}h^{b}
-g^{(q)}_{ac}(\xi^{(p)})g^{(r)cd}_{m}(\theta^{(p)})g^{(q)}_{db}(\xi^{(p)})h^{a}h^{b}
\nonumber\\
&=&
g^{(p)}_{ac}(\xi^{(p)})g^{(r)cd}_{m}(\theta^{(p)})g^{(q)}_{bd}(\xi^{(p)})h^{a}h^{b}
\nonumber\\
&=&
g^{(p)}_{ac}(\xi^{(p)})g^{(r)cd}_{m}(\theta^{(p)})g^{(q)}_{bd}(\xi^{(p)})
\left\{
	-g^{(p)ae}_{m}(\theta^{(p)})g^{(p)}_{e\kappa}(\xi^{(p)})h^{\kappa}+\mathrm{o}(1)
\right\}
\left\{
	-g^{(p)bf}_{m}(\theta^{(p)})g^{(p)}_{f\lambda}(\xi^{(p)})h^{\lambda}+\mathrm{o}(1)
\right\}
\nonumber\\
&=&
g^{(p)}_{a\kappa}(\xi^{(p)})g^{(r)ab}_{m}(\theta^{(p)})g^{(q)}_{bc}(\xi^{(p)})g^{(p)cd}_{m}(\theta^{(p)})
g^{(p)}_{d\lambda}(\xi^{(p)})h^{\kappa}h^{\lambda}+\mathrm{o}(N)
\nonumber\\
&=&
g^{(p)}_{a\kappa}(\xi^{(p)})g^{(r)ab}_{m}(\theta^{(p)})
\{g^{(r)}_{bc}(\xi^{(p)})-g^{(p)}_{bc}(\xi^{(p)})\}
g^{(p)cd}_{m}(\theta^{(p)})
g^{(p)}_{d\lambda}(\xi^{(p)})h^{\kappa}h^{\lambda}+\mathrm{o}(N)
\nonumber\\
&=&
g^{(p)}_{a\kappa}(\xi^{(p)})g^{(p)ab}_{m}(\theta^{(p)})g^{(p)}_{b\lambda}(\xi^{(p)})h^{\kappa}h^{\lambda}
-
g^{(p)}_{a\kappa}(\xi^{(p)})g^{(r)ab}_{m}(\theta^{(p)})g^{(p)}_{b\lambda}(\xi^{(p)})h^{\kappa}h^{\lambda}
+\mathrm{o}(N).
\label{S_ab}
\end{eqnarray}
From (\ref{h_decomp}) and the relationship that $g^{(q)}=g^{(r)}-g^{(p)}$,
we have
\begin{eqnarray}
&{}&\hspace{-7mm}S_{a\kappa}h^{a}h^{\kappa}
\nonumber\\
&=&g^{(p)}_{ac}(\xi^{(p)})g^{(r)cd}_{m}(\theta^{(p)})g^{(q)}_{d\kappa}(\xi^{(p)})h^{a}h^{\kappa}
\nonumber\\
&=&g^{(p)}_{ac}(\xi^{(p)})g^{(r)cd}_{m}(\theta^{(p)})g^{(q)}_{d\kappa}(\xi^{(p)})
\left\{
	-g^{(p)ae}_{m}(\theta^{(p)})g^{(p)}_{e\lambda}(\xi^{(p)})h^{\lambda}+\mathrm{o}(1)
\right\}h^{\kappa}
\nonumber\\
&=&
-g^{(p)}_{ac}(\xi^{(p)})g^{(r)cd}_{m}(\theta^{(p)})
\left\{g^{(r)}_{d\kappa}(\xi^{(p)})-g^{(p)}_{d\kappa}(\xi^{(p)})\right\}
g^{(p)ae}_{m}(\theta^{(p)})g^{(p)}_{e\lambda}(\xi^{(p)})h^{\lambda}h^{\kappa}+\mathrm{o}(N)
\nonumber\\
&=&-g^{(p)}_{a\kappa}(\xi^{(p)})g^{(r)ab}_{m}(\theta^{(p)})g^{(r)}_{b\lambda}(\xi^{(p)})h^{\kappa}h^{\lambda}
\nonumber\\
&&+g^{(p)}_{a\kappa}(\xi^{(p)})g^{(r)ab}_{m}(\theta^{(p)})g^{(p)}_{b\lambda}(\xi^{(p)})h^{\kappa}h^{\lambda}
+\mathrm{o}(N)
\label{S_akappa}
\end{eqnarray}
and
\begin{eqnarray}
&{}&\hspace{-7mm}S_{\kappa\lambda}h^{\kappa}h^{\lambda}
\nonumber\\
&=&
\left[
	g^{(q)}_{\kappa\lambda}(\xi^{(p)})
	-g^{(q)}_{\kappa a}(\xi^{(p)})g^{(r)ab}_{m}(\theta^{(p)})g^{(q)}_{b\lambda}(\xi^{(p)})
\right]h^{\kappa}h^{\lambda}
\nonumber\\
&=&\left\{
	g^{(q)}_{\kappa\lambda}(\xi^{(p)})
	-\left\{g^{(r)}_{\kappa a}(\xi^{(p)})-g^{(p)}_{\kappa a}(\xi^{(p)})\right\}
	g^{(r)ab}_{m}(\theta^{(p)})
	\left\{g^{(r)}_{b\lambda}(\xi^{(p)})-g^{(p)}_{b \lambda}(\xi^{(p)})\right\}
\right\}h^{\kappa}h^{\lambda}
\nonumber\\
&=&
g^{(q)}_{\kappa\lambda}(\xi^{(p)})h^{\kappa}h^{\lambda}
-g^{(r)}_{\kappa a}(\xi^{(p)})g^{(r)ab}_{m}(\theta^{(p)})g^{(r)}_{b\lambda}(\xi^{(p)})h^{\kappa}h^{\lambda}
-g^{(p)}_{\kappa a}(\xi^{(p)})g^{(r)ab}_{m}(\theta^{(p)})g^{(p)}_{b\lambda}(\xi^{(p)})h^{\kappa}h^{\lambda}
\nonumber\\
&&+2g^{(p)}_{\kappa a}(\xi^{(p)})g^{(r)ab}_{m}(\theta^{(p)})g^{(r)}_{b\lambda}(\xi^{(p)})h^{\kappa}h^{\lambda}.
	\label{S_kappalambda}
\end{eqnarray}

From (\ref{S_ab}), (\ref{S_akappa}), and (\ref{S_kappalambda}),
we obtain the following equation:
\begin{eqnarray*}
S_{\alpha\beta}h^{\alpha}h^{\beta}&=&
	\left\{g^{(q)}_{\kappa\lambda}(\xi^{(p)})
	+g^{(p)}_{a\kappa}(\xi^{(p)})g^{(p)}_{m}(\theta^{(p)})g^{(p)}_{b\lambda}(\xi^{(p)})
	-g^{(r)}_{a\kappa}(\xi^{(p)})g^{(r)}_{m}(\theta^{(p)})g^{(r)}_{b\lambda}(\xi^{(p)})
\right\}
h^{\kappa}h^{\lambda}+\mathrm{o}(N).
\end{eqnarray*}
Thus, we obtain (\ref{PS_equivalence}) and complete the proof of Theorem \ref{Risk_decomp}.
\end{proof}

\begin{proof}[Proof of Theorem \ref{Unbiased}]
Since
$\hat{h}^{\alpha}$ is decomposed as
\begin{eqnarray}
\frac{\hat{h}^{\alpha}}{\sqrt{N}}&=&\hat{\xi}^{\alpha}(x^{(N)})-\xi^{*\alpha}
+\xi^{*\alpha}-\xi^{(p)\alpha}
+\xi^{(p)\alpha}-\xi^{\alpha}(\hat{\theta}(x^{(N)}),0)
\nonumber\\
&=&g^{(p)\alpha\beta}(\xi^{*})\partial_{\beta}\log p(x^{(N)}|\xi^{*})
+\frac{h^{\alpha}}{\sqrt{N}}
\nonumber\\
&&-\delta^{\alpha}_{a}g^{(p)ab}_{m}(\theta^{(p)})\partial_{b}\log p(x^{(N)}|\omega^{(p)})
+\mathrm{O}(1/N),
\label{est_h_decomp}
\end{eqnarray}
the expectation of $\hat{S}_{\alpha\beta}\hat{h}^{\alpha}\hat{h}^{\beta}/N$
is given as
\begin{eqnarray*}
&{}&\hspace{-7mm}\mathrm{E}_{\omega^{*}}[\hat{S}_{\alpha\beta}\hat{h}^{\alpha}\hat{h}^{\beta}]/N
\nonumber\\
&=&\mathrm{E}_{\omega^{*}}[S_{\alpha\beta}\hat{h}^{\alpha}\hat{h}^{\beta}]/N+\mathrm{o}(1)
\nonumber\\
&=&\mathrm{E}_{\omega^{*}}\left[
S_{\alpha\beta}
\left\{
g^{(p)\alpha\gamma}(\xi^{*})\partial_{\gamma}\log p(x^{(N)}|\xi^{*})
+\frac{h^{\alpha}}{\sqrt{N}}
-\delta^{\alpha}_{a}g^{(p)ac}_{m}(\theta^{(p)})\partial_{c}\log p(x^{(N)}|\omega^{(p)})
\right\}
\right.
\nonumber\\
&&
\left.\quad\quad
\times\left\{
g^{(p)\beta\delta}(\xi^{*})\partial_{\delta}\log p(x^{(N)}|\xi^{*})
+\frac{h^{\beta}}{\sqrt{N}}
-\delta^{\beta}_{b}g^{(p)bd}_{m}(\theta^{(p)})\partial_{d}\log p(x^{(N)}|\omega^{(p)})
\right\}
\right]
\nonumber\\
&&+\mathrm{o}(1)
\nonumber\\
&=&S_{\alpha\beta}\frac{h^{\alpha}h^{\beta}}{N}+S_{\alpha\beta}g^{(p)\alpha\beta}(\xi^{*})
+S_{ab}g^{(p)ab}_{m}(\theta^{(p)})
-2S_{ab}g^{(p)ab}_{m}(\theta^{(p)})+\mathrm{o}(1)
\nonumber\\
&=&S_{\alpha\beta}\frac{h^{\alpha}h^{\beta}}{N}+S_{\alpha\beta}g^{(p)\alpha\beta}(\xi^{*})
-S_{ab}g^{(p)ab}_{m}(\theta^{(p)})+\mathrm{o}(1).
\end{eqnarray*}
Thus, we complete the proof.
\end{proof}

\bibliographystyle{Yano_sjs}
\bibliography{Yano_Komaki_MSPIC}

\begin{thebibliography}{18}
\expandafter\ifx\csname natexlab\endcsname\relax\def\natexlab#1{#1}\fi

\bibitem[{Akaike(1973)}]{Akaike(1973)}
Akaike, H. (1973).
\newblock Information theory and an extension of the maximum likelihood
  principle.
\newblock In B.~Petrov and F.~Caski, eds., \emph{Proc. of the 2nd international
  symposium of information theory}. Akadimiai Kiado, pp. 267--281.

\bibitem[{Akaike(1980)}]{Akaike(1980)}
Akaike, H. (1980).
\newblock On the use of the predictive likelihood of a {G}aussian model.
\newblock \emph{Ann. Inst. Statist. Math.} \textbf{{\bf 32}}, \,pp. 311--324.

\bibitem[{Claeskens and Hjort(2003)}]{ClaeskensandHjort(2003)}
Claeskens, G. and Hjort, N.~L. (2003).
\newblock The focused information criterion.
\newblock \emph{J. Amer. Statist. Assoc.} \textbf{{\bf 98}}, \,pp. 900--916.

\bibitem[{Ghosh \emph{{\rm et}~{\rm al.}}(2006)Ghosh, Delampady, and
  Samanta}]{Ghosh_Introduction_Bayes}
Ghosh, J.~K., Delampady, M., and Samanta, T. (2006).
\newblock \emph{An {I}ntroduction to {B}ayesian {A}nalysis {T}heory and
  {M}ethods}.
\newblock Springer Science+Business Media, New York.

\bibitem[{Hartigan(1998)}]{Hartigan(1998)}
Hartigan, J. (1998).
\newblock The maximum likelihood prior.
\newblock \emph{Ann. Statist.} \textbf{{\bf 26}}, \,pp. 2083--2103.

\bibitem[{Hjort and Claeskens(2003)}]{HjortandClaeskens(2003)}
Hjort, N.~L. and Claeskens, G. (2003).
\newblock Frequentist model average estimators.
\newblock \emph{J. Amer. Statist. Assoc.} \textbf{{\bf 98}}, \,pp. 879--899.

\bibitem[{Kitagawa(1997)}]{Kitagawa(1997)}
Kitagawa, G. (1997).
\newblock Information criteria for the predictive evaluation of {B}ayesian
  models.
\newblock \emph{Comm. Statist. Theory Methods} \textbf{{\bf 26}}, \,pp.
  2223--2246.

\bibitem[{Komaki(1996)}]{Komaki(1996)}
Komaki, F. (1996).
\newblock On asymptotic properties of predictive distributions.
\newblock \emph{Biometrika} \textbf{{\bf 83}}, \,pp. 299--313.

\bibitem[{Komaki(2015)}]{Komaki(2014)}
Komaki, F. (2015).
\newblock Asymptotic properties of {B}ayesian predictive densities when the
  distributions of data and target variables are different.
\newblock \emph{Bayesian Anal.} \textbf{{\bf 10}}, \,pp. 31--51.

\bibitem[{Konishi and Kitagawa(2003)}]{KonishiandKitagawa(2003)}
Konishi, S. and Kitagawa, G. (2003).
\newblock Asymptotic theory for information criteria in model
  selection--functional approach.
\newblock \emph{J. Statist. Plann. and Infer} \textbf{{\bf 114}}, \,pp. 45--61.

\bibitem[{Leeb and P{\"o}tscher(2005)}]{LeebandPotscher(2005)}
Leeb, H. and P{\"o}tscher, B.~M. (2005).
\newblock Model selection and inference: Facts and fiction.
\newblock \emph{Econom. Theory} \textbf{{\bf 21}}, \,pp. 21--59.

\bibitem[{Liang and Barron(2004)}]{LiangandBarron(2004)}
Liang, F. and Barron, A. (2004).
\newblock Exact minimax strategies for predictive density estimation, data
  compression, and model selection.
\newblock \emph{IEEE TRAN. ON INFOR. THEORY} \textbf{{\bf 50}}, \,pp.
  2708--2726.

\bibitem[{Lv and Liu(2014)}]{LvandLiu(2014)}
Lv, L. and Liu, J. (2014).
\newblock Model selection principles in misspecified models.
\newblock \emph{J. R. Statrist. Soc. B} \textbf{{\bf 76}}, \,pp. 141--167.

\bibitem[{Sei and Komaki(2007)}]{SeiandKomaki(2007)}
Sei, T. and Komaki, F. (2007).
\newblock Bayesian prediction and model selection for locally asymptotically
  mixed normal models.
\newblock \emph{J. Statist. Plann. and Infer} \textbf{{\bf 137}}, \,pp.
  2523--2534.

\bibitem[{Shibata(1981)}]{Shibata(1981)}
Shibata, R. (1981).
\newblock An optimal selection of regression variables.
\newblock \emph{Biometrika} \textbf{{\bf 68}}, \,pp. 45--54.

\bibitem[{Shimodaira(1997)}]{Shimodaira(1997)}
Shimodaira, H. (1997).
\newblock Assessing the error probability of the model selection test.
\newblock \emph{Ann. Inst. Statist. Math.} \textbf{{\bf 49}}, \,pp. 395--410.

\bibitem[{Takeuchi(1976)}]{Takeuchi(1976)}
Takeuchi, K. (1976).
\newblock Distribution of information statistics and a criterion of model
  fitting.
\newblock \emph{Suri-Kagaku} \textbf{{\bf 153}}, \.pp. 12--18.
\newblock In Japanese.

\bibitem[{van~der Vaart(1998)}]{vanderVaart}
van~der Vaart (1998).
\newblock \emph{Asymptotic statistics}.
\newblock Cambridge University Press, New York.

\end{thebibliography}
\end{document}